\documentclass[12pt]{article}
\usepackage{amsmath}
\usepackage{mathtools}
\usepackage{amsfonts}
\usepackage{amsthm}
\usepackage[english]{babel} 
\usepackage[iso-8859-7]{inputenc}
\usepackage{bold-extra}
\usepackage{indentfirst}
\usepackage{hyperref}
\usepackage[margin=1.1in]{geometry}
\newtheorem{definition}{Definition}
\newtheorem{lemma}[definition]{Lemma}
\newtheorem{theorem}[definition]{Theorem}
\newtheoremstyle{r}
{}
{}
{\normalfont}
{}
{\scshape}
{.}
{ }
{}
\theoremstyle{r}
\newtheorem{remark}[definition]{Remark}
\numberwithin{definition}{section}
\allowdisplaybreaks
\title{Universal Taylor Series in several variables depending on parameters}
\author{G. Gavrilopoulos, K. Maronikolakis and V. Nestoridis}
\date{}
\begin{document}
	\maketitle
	\begin{abstract}
		We establish generic existence of Universal Taylor Series on products $\Omega = \prod \Omega_i$ of planar simply connected domains $\Omega_i$ where the universal approximation holds on products $K$ of planar compact sets with connected complements provided $K \cap \Omega = \emptyset$. These classes are with respect to one or several centers of expansion and the universal approximation is at the level of functions or at the level of all derivatives. Also, the universal functions can be smooth up to the boundary, provided that $K \cap \overline{\Omega} = \emptyset$ and $\{\infty\} \cup [\mathbb{C} \setminus \overline{\Omega}_i]$ is connected for all $i$. All previous kinds of universal series may depend on some parameters; then the approximable functions may depend on the same parameters, as it is shown in the present paper. These universalities are topologically and algebraically generic.
	\end{abstract}
	\textbf{AMS classification numbers:} 32A05, 30K05, 30E10.
	\section{Introduction}
	In \cite{Bayart} and the references therein, one can read about the history of universal Taylor series. The first universal Taylor series was established before 1914 by Fekete \cite{Pal}. The existence of a power series $\sum\limits_{k = 1}^\infty a_k x^k, a_k \in \mathbb{R}$, such that, its partial sums $\sum\limits_{k = 1}^n a_k x^k, n = 1, 2, \dots$ approximate uniformly on $[-1, 1]$ every continuous function $h : [-1, 1] \rightarrow \mathbb{R}$ with $h(0) = 0$ was shown.
	
	In 1945, we have the universal trigonometric series in the sense of Menchoff \cite{Menshov}. There exists a trigonometric series $\sum\limits_{k = - \infty}^\infty a_k e^{ikx}, a_k \in \mathbb{C}, a_k \to 0 \text{ as } |k| \to + \infty$, such that its partial sums $\sum\limits_{k = - n}^n a_k e^{ikx}, n = 0, 1, 2, \dots$ approximate in the almost everywhere sense every $2 \pi$ - periodic complex measurable function $h : \mathbb{R} \rightarrow \mathbb{C}$.
	
	In 1951, Seleznev \cite{Seleznev} showed the existence of a power series $\sum\limits_{k = 0}^\infty a_k z^k, a_k \in \mathbb{C}$, with zero radius of convergence, such that its partial sums $\sum\limits_{k = 0}^n a_k z^k, n = 0, 1, 2, \dots$ approximate uniformly any polynomial on any compact set $K \subseteq \mathbb{C} \setminus \{0\}$ with connected complement $\mathbb{C} \setminus K$.
	
	In the early 70's, W. Luh \cite{Luh} and independently Chui and Parnes \cite{Chui} showed the existence of power series $\sum\limits_{k = 0}^\infty a_k z^k, a_k \in \mathbb{C}$ with radius of convergence $1$, such that its partial sums approximate any polynomial on any compact set $K \subseteq \{z \in \mathbb{C} : |z| > 1\}$ with connected complement. Grosse - Erdmann showed in his thesis (1987) \cite{Grosse1}, using Baire's Theorem, that the phenomenon is generic in the space $H(D)$ of all holomorphic functions $f$ on the open unit disk $D = \{z \in \mathbb{C} : |z| < 1\}$ endowed with the topology of uniform convergence on compacta.
	
	In 1996, V. Nestoridis strengthened the result of Luh and Chui and Parnes allowing the compact set $K$ to meet the boundary of the unit disk \cite{Nestoridis1}. Soon, the open unit disk $D$ was replaced by any simply connected domain $\Omega$ in $\mathbb{C}$ and the set of functions whose Taylor series realized the previously mentioned approximations, called universal Taylor series, has been proven to be a $G_\delta$ and dense subset of the space $H(\Omega)$.
	
	Roughly speaking, it is well known that under some assumptions, the partial sums of the Taylor expansion of a function $f$ near the center of expansion approximate only $f$. But what happens far from the center of expansion? The previously mentioned results show that generically, they approximate all functions that can be approximated by polynomials.
	
	More recently \cite{Nestoridis2}, it was shown that not only the polynomials $P$ could be approximated on the compact set $K$ by the partials sums $S_{N_k}, k = 1, 2, \dots$ but also every derivative of $P$ can be approximated on $K$ by the corresponding derivative of $S_{N_k}$.
	
	In \cite{Abakumov} we consider families of universal Taylor series depending on a parameter; then the function $h$ to be approximated by the partial sums can depend on the same parameter. This led to functions of several complex variables; see also \cite{Clouatre}, \cite{Daras}, \cite{Kioulafa}.
	
	In the present paper we obtain several extensions of the result in \cite{Abakumov} and other results in the case of several complex variables. We divide the set of variables in two groups: one group, the set of parameters and a second group, the set of variables appearing in the Taylor expansion. We also obtain that the universal functions can be smooth on the boundary, provided that the sets $K$ are disjoint from the closure of the domain of definition $\Omega = \prod \Omega_i$ and that the planar simply connected domain $\Omega _i$ satisfies that $\{\infty\} \cup [\mathbb{C} \setminus \overline{\Omega}_i]$ is connected; see also \cite{Kariofillis}.
	
	We close mentioning that the set $E$ of functions which can be approximated uniformly on the compact set $K$ by polynomials is of interest for the theory of universal Taylor series. If $K \subseteq \mathbb{C}$, there exists the celebrated theorem of Mergelyan (1952) stating that, if $\mathbb{C} \setminus K$ is connected, then this set $E$ coincides with $A(K) = \{h : K \rightarrow \mathbb{C}, \text{ continuous on } K \text{ holomorphic on } K^{\mathrm{o}}\}$. In several complex variables, the theory is much less developed. It is true that the set $E$ is included in $A(K)$, but there is a smaller function algebra $A_D(K), E \subseteq A_D(K) \subseteq A(K)$, which is more appropriate for approximation, \cite{Falco}.
	
	In \cite{Falco} some sufficient conditions so that $E = A_D(K)$ are given, when $K$ is a product of compact planar sets. In the present paper we are using this new algebra $A_D(K)$ and we give its definition in section 2. 
	
	The organization of the paper is as follows. In section 2 we give some preliminary results concerning $A_D(K)$. In section 3 we extend to several variables \cite{Abakumov}, \cite{Kioulafa} and \cite{Nestoridis1}. In section 4 we strengthen the result of section 2 obtaining approximation for all derivatives. Finally, in section 5 we extend the result of \cite{Kariofillis} and we obtain smooth universal Taylor series in the weak sense.
	
	Our method of proof is based on Baire's Category Theorem. For its use in analysis we refer to \cite{Kahane} and \cite{Grosse2}.
	\section{Preliminaries}
	We present some context from \cite{Falco} which will be useful later on.
	\begin{definition}\cite{Falco}
		Let $L \subset \mathbb{C}^n$ be a compact set. A function $f : L \rightarrow \mathbb{C}$ is said to belong to the class $A_D(L)$ if it is continuous on $L$ and, for every open disk $D \subset \mathbb{C}$ and every injective mapping $\phi : D \rightarrow L \subset \mathbb{C}^n$ holomorphic on $D$, the composition $f \circ \phi : D \rightarrow \mathbb{C}$ is holomorphic on D.
	\end{definition}
	We recall that a function $\phi : D \rightarrow \mathbb{C}^n$, where $D$ is a planar domain, is holomorphic if each coordinate is a holomorphic function.
	
	We have the following approximation lemmas:
	\begin{lemma}\cite{Falco}
		\label{Merg_1}
		Let $L = \displaystyle\prod_{i = 1}^n L_i$ where the sets $L_i$ are compact subsets of $\mathbb{C}$ with connected complement for $i = 1, 2, \dots, n$. If $f \in A_D(L)$ and $\varepsilon > 0$, then there exists a polynomial $p$ such that $\sup\limits_{z \in L}|f(z) - p(z)| < \varepsilon$.
	\end{lemma}
	\begin{lemma}\cite{Falco}
		\label{Merg_2}
		Let $L = \displaystyle\prod_{i = 1}^n L_i$ where the sets $L_i$ are compact subsets of $\mathbb{C}$ with connected complement for $i = 1, 2, \dots, n$ and $I$ a finite subset of $\mathbb{N}^n$. If $f$ is holomorphic on a neighbourhood of $L$ and $\varepsilon > 0$, then there exists a polynomial $p$ such that $$\sup\limits_{z \in L}\Big|\dfrac{\partial^{m_1 + \dots + m_n}}{\partial z_1^{m_1} \cdots \partial z_n^{m_n}}(f - p)(z)\Big| < \varepsilon$$ for all $m = (m_1, \dots, m_n) \in I$.
	\end{lemma}
	Lemma \ref{Merg_1} is part of a result in \cite{Falco}, Remark 4.9 (3). For Lemma \ref{Merg_2} see Proposition 2.8 of \cite{Falco}. We also have the following (Proposition 2.7 of \cite{Falco}):
	\begin{lemma}\cite{Falco}
		\label{Simply}
		Let $U_i \subseteq \mathbb{C}, i = 1, \dots, n$ be simply connected domains and $U = \displaystyle\prod_{i = 1}^n U_i$. Then, the set of polynomials of $n$ variables is dense in $H(U)$ endowed with the topology of uniform convergence on all compact subsets of $U$.
	\end{lemma}
	We now introduce some notation:
	
	Suppose that $m = (m_1, \dots, m_d) \in \mathbb{N}^d$ and $z = (z_1, \dots, z_d) \in \mathbb{C}^d$, The,n we denote the monomial $z_1^{m_1} \cdot \dots \cdot z_d^{m_d}$ by $z^m$. Let $G \subseteq \mathbb{C}^r, \Omega \subseteq \mathbb{C}^d$ be open sets and $f : G \times \Omega \rightarrow \mathbb{C}$ be a holomorphic function, then, for each $w \in G$ and $\zeta \in \Omega$, we set $$\gamma_m(f,w,\zeta) = \dfrac{1}{m_1! \cdots m_d!} \dfrac{\partial^{m_1 + \dots + m_d}}{\partial u_1^{m_1} \cdots \partial u_d^{m_d}}f(w, u)\big|_{u = \zeta}$$where $u = (u_1, \dots, u_d)$. Now, let $N_k, k = 0, 1, \dots$ be an enumeration of $\mathbb{N}^d$ then for each $n = 0, 1, \dots$ we set $S_n(f, w, \zeta)(z) = \sum\limits_{k = 0}^{n}a_k(f, w, \zeta)(z - \zeta)^{N_k}$ where $a_k(f, w, \zeta) \equiv \gamma_{N_k}(f, w, \zeta)$.
	\section{Universal Taylor series with parameters in $H(G \times \Omega)$}
	Let $G_i \subseteq \mathbb{C}, i = 1, \dots, r$ and $\Omega_i \subseteq \mathbb{C}, i = 1, \dots, d$ be simply connected domains. We set $G = \displaystyle\prod_{i = 1}^{r} G_i$ and $\Omega = \displaystyle\prod_{i = 1}^{d} \Omega_i$.
	\begin{definition}
		Let $ N_k, k = 0, 1, \dots$ be an enumeration of $\mathbb{N}^d$ and $\mu$ be an infinite subset of $\mathbb{N}$. We define $U^\mu(G, \Omega)$ to be the set of $f \in H(G \times \Omega)$ such that for every set $K = \displaystyle\prod_{i = 1}^{d} K_i$ where $K_i \subseteq \mathbb{C}$ are compact with $\mathbb{C} \setminus K_i$ connected for all $i = 1, \dots d$ and there is at least one $i_0 \in \{1, \dots, d\}$ such that $\Omega_{i_0} \bigcap K_{i_0} = \emptyset$ and every continuous function $h : G \times K \rightarrow \mathbb{C}$ such that for every fixed $w \in G$, the function $K \ni z \rightarrow h(w, z) \in \mathbb{C}$ belongs to $A_D(K)$ and for every fixed $z \in K$, the function $G \ni w \rightarrow h(w, z) \in \mathbb{C}$ belongs to $H(G)$, there exists a strictly increasing sequence $\lambda_n \in \mu, n = 1, 2, \dots$ such that 
		\begin{align*}	
		& \sup_{\zeta \in M}\sup_{w \in L}\sup_{z \in K}\big|S_{\lambda_n}(f, w, \zeta)(z) - h(w,z)\big| \longrightarrow 0 \text{ and } \\
		& \sup_{\zeta \in M}\sup_{w \in L}\sup_{z \in M}\big|S_{\lambda_n}(f, w, \zeta)(z) - f(w,z)\big| \longrightarrow 0
		\end{align*}
		for every compact subset $L$ of $G$ and for every compact subset of $M$ of $\Omega$.
	\end{definition}
	We notice that the class $U^\mu(G, \Omega)$ can also be defined without the requirement that the sequence  $(\lambda_n)_{n \in \mathbb{N}}$ is strictly increasing and then the two definitions are equivalent; see \cite{Vlachou}.
	
	We also have the following equivalence: A continuous function $h : G \times K \rightarrow \mathbb{C}$ has the properties that for every fixed $w \in G$, the function $K \ni z \rightarrow h(w, z) \in \mathbb{C}$ belongs to $A_D(K)$ and for every fixed $z \in K$, the function $G \ni w \rightarrow h(w, z) \in \mathbb{C}$ belongs to $H(G)$ if and only if for every compact sets $L_i \subseteq G_i$ with $\mathbb{C} \setminus L_i$ connected for all $i = 1, \dots, r$, the restriction of $h$ to the Cartesian product $\big(\displaystyle\prod_{i = 1}^r L_i\big) \times K$ belongs to $A_D\Big[\big(\displaystyle\prod_{i = 1}^r L_i\big) \times K\Big]$. This holds, because, as it is proven in \cite{Falco}, a continuous complex function defined on a product $M = \displaystyle\prod_{l = 1}^{\sigma}M_l$ of planar compact sets $M_l$ belongs to $A_D(M)$ if and only if, for every $l \in \{1, \dots, \sigma\}$, the corresponding slice functions belong to $A_D(M_l) = A(M_l).$
	
	Let $F_\tau, \tau = 1, 2, \dots$ be an exhausting family of compact subsets of $G$ and $M_p, p = 1, 2, \dots$ be an exhausting family of compact subsets of $\Omega$, where each one of the sets $F_\tau$ and $M_p$ is a product of planar compact sets with connected complement. We may also assume that $F_\tau \subseteq F_{\tau + 1}$ for all $\tau = 1, 2, \dots$ and $M_p \subseteq M_{p + 1}$ for all $p = 1, 2, \dots$. It is known that, for each $i = 1, \dots, d$, there exist compact sets $R_{i, j} \subseteq \mathbb{C} \setminus\Omega_i, j = 1, 2, \dots$ with connected complement, such that for every compact set $T \subseteq \mathbb{C} \setminus \Omega_i$ with connected complement, there exists an integer $j$ such that $T \subseteq R_{i, j}$. Let $T_m$ be an enumeration of all $\displaystyle\prod_{i = 1}^{d} Q_i$, such that there exists an integer $i_0 = 1, 2, \dots, d$ with $Q_{i_0} \in \{R_{i_0, j} : j = 1, 2, \dots,\}$ and the rest of the sets $Q_i$ are closed disks centered at 0 whose radius is a positive integer.
	
	Let $f_j, j = 1, 2, \dots$ be an enumeration of all polynomials of $r + d$ variables having coefficients with rational coordinates. For any $\tau, p, m, j, s, n$ with $\tau, p, m, j, s \geq 1, n \geq 0$, we denote by $E(\tau, p, m, j, s, n)$ the set $$\Big\{f \in H(G \times \Omega) : \sup_{\zeta \in M_p}\sup_{w \in F_\tau}\sup_{z \in T_m}\big|S_n(f, w, \zeta)(z) - f_j(w,z)\big| < \frac{1}{s}\Big\},$$ and by $F(\tau, p, s, n)$ the set
	$$\Big\{f \in H(G \times \Omega) : \sup_{\zeta \in M_p}\sup_{w \in F_\tau}\sup_{z \in M_p}\big|S_n(f, w, \zeta)(z) - f(w,z)\big| < \frac{1}{s}\Big\}.$$
	\begin{lemma}
	\label{Inter_1}
		With the above assumptions,
		$$U^{\mu}(G, \Omega)=\bigcap_{\tau, p, m, j, s} \ \ \bigcup_{n\in \mu} \big[E(\tau, p, m, j, s, n) \bigcap F(\tau, p, s, n)\big].$$
	\end{lemma}
	\begin{proof}
		For the inclusion $U^{\mu}(G, \Omega) \subseteq \bigcap\limits_{\tau, p, m, j, s} \ \ \bigcup\limits_{n\in \mu} \big[E(\tau, p, m, j, s, n) \bigcap F(\tau, p, s, n)\big]$, we consider a function $f\in U^{\mu}(G, \Omega)$ and arbitrary positive integers $\tau, p, m, j, s$.
		We shall show that there exists an integer $n' \in \mu$ such that $f \in E(\tau, p, m, j, s, n') \bigcap F(\tau, p, s, n')$. Since $f_j$ is a polynomial, it is holomorphic on $\mathbb{C}^r \times \mathbb{C}^d.$ Thus, using the definition of the class $U^{\mu}(G, \Omega)$, we fix a sequence $(\lambda_n)_{n \in \mathbb{N}},\lambda_n\in \mu$, such that 
		$$\sup_{\zeta \in M_p}\sup_{w \in F_\tau}\sup_{z \in T_m}\big|S_{\lambda_n}(f, w, \zeta)(z) - f_j(w,z)\big| \to 0 \ \text{as} \ n \to \infty$$ 
		and 
		$$\sup_{\zeta \in M_p}\sup_{w \in F_\tau}\sup_{z \in M_p}\big|S_{\lambda_n}(f, w, \zeta)(z) - f(w,z)\big| \to 0 \ \text{as} \ n \to \infty \ .$$ 
		Thus, there exist positive integers $n_1, n_2$, such that 
		$$\sup_{\zeta \in M_p}\sup_{w \in F_\tau}\sup_{z \in T_m}\big|S_{\lambda_n}(f, w, \zeta)(z) - f_j(w,z)\big|<\frac{1}{s}$$ for every $n \geq n_1$ and $$\sup_{\zeta \in M_p}\sup_{w \in F_\tau}\sup_{z \in M_p}\big|S_{\lambda_n}(f, w, \zeta)(z) - f(w,z)\big| < \frac{1}{s}$$ for every $n \geq n_2$. By setting $n_0 = \max\{n_1, n_2\}$ and $n' = \lambda_{n_0}$ we get $f\in $\break$E(\tau, p, m, j, s, n') \bigcap F(\tau, p, s, n')$ which is exactly what we wanted to show.
		
		For the other inclusion, we consider a function 
		$$f\in \bigcap\limits_{\tau, p, m, j, s} \ \ \bigcup\limits_{n\in \mu} \big[E(\tau, p, m, j, s, n) \bigcap F(\tau, p, s, n)\big] \ .$$
		Let $K = \displaystyle\prod_{i = 1}^{d} K_i$ where $K_i \subseteq \mathbb{C}$ are compact with $\mathbb{C} \setminus K_i$ connected for all $i = 1, \dots d$ and there is at least one $i_0 \in \{1, \dots, d\}$ such that $\Omega_{i_0} \bigcap K_{i_0} = \emptyset$. Let also $h : G \times K \rightarrow \mathbb{C}$ be a continuous function such that for every fixed $w \in G$, the function $K \ni z \rightarrow h(w, z) \in \mathbb{C}$ belongs to $A_D(K)$ and for every fixed $z \in K$, the function $G \ni w \rightarrow h(w, z) \in \mathbb{C}$ belongs to $H(G)$. 
		Then, for each positive integer $\tau$ we have $h\in A_D(F_{\tau} \times K)$; thus, by Lemma \ref{Merg_1}, there exists a positive integer $j_\tau$, such that 
		$$\sup_{w\in F_{\tau}}\sup_{z\in K}\left|h(w, z)-f_{j_\tau}(w, z)\right|<\frac{1}{2\tau} \ .$$
		Since the above relation does not depend on $\zeta$, it can be rewritten as
		$$\sup_{\zeta \in M_\tau}\sup_{w\in F_{\tau}}\sup_{z\in K}\left|h(w, z)-f_{j_\tau}(w, z)\right|<\frac{1}{2\tau} \ . \ (1)$$
		By the definition of the sets $T_m$, there exists a positive integer $m_0$, such that $K \subseteq T_{m_0}$. For each positive integer $\tau$ we have $f \in \bigcup\limits_{n\in \mu} \big[E(\tau, \tau, m_0, j_\tau, 2 \tau, n) \bigcap F(\tau, \tau, 2 \tau, n)\big]$, thus, there exists a positive integer $\lambda_{\tau}\in \mu$, such that 
		$$\sup_{\zeta\in M_{\tau}}\sup_{w\in F_{\tau}}\sup_{z\in T_{m_0}} \left|S_{\lambda_{\tau}}(f,w,\zeta)(z)-f_{j_\tau}(w,z)\right|<\frac{1}{2\tau}$$ 
		and $$\sup_{\zeta\in M_{\tau}}\sup_{w\in F_{\tau}}\sup_{z\in M_{\tau}} \left|S_{\lambda_{\tau}}(f,w,\zeta)(z)-f(w,z)\right|<\frac{1}{2\tau} \ . \ (2)$$
		Since $K \subset T_{m_0}$, it follows that 
		$$\sup_{\zeta\in M_{\tau}}\sup_{w\in F_{\tau}}\sup_{z\in K} \left|S_{\lambda_{\tau}}(f,w,\zeta)(z)-f_{j_\tau}(w,z)\right|<\frac{1}{2\tau} \ . \ (3)$$
		From $(1)$ and $(3)$ it follows that for each positive integer $\tau$, we have
		$$\sup_{\zeta\in M_{\tau}}\sup_{w\in F_{\tau}}\sup_{z\in K} \left|S_{\lambda_{\tau}}(f,w,\zeta)(z)-h(w,z)\right|<\frac{1}{\tau} \ . \ (4)$$
		
		We consider the sequence $(\lambda_\tau)_{\tau \in \mathbb{N}}$.
		Let $L \subseteq G$ and $M \subseteq \Omega$ be compact sets.
		Then, there exists a positive integer $\tau_0$, such that $L \subseteq F_{\tau_0}$ and $M \subseteq M_{\tau_0}$.
		Since the families $\{F_{\tau}:\tau=1,2,\ldots\}$ and $\{M_p:p=1,2,\ldots\}$ are increasing, it follows that $L\subseteq F_{\tau}$ and $M\subseteq M_{\tau}$ for any positive integer $\tau\geq \tau_0$, therefore, from $(4)$, it follows that 
		$$\sup_{\zeta\in M}\sup_{w\in L}\sup_{z\in K} \left|S_{\lambda_{\tau}}(f,w,\zeta)(z)-h(w,z)\right|<\frac{1}{\tau} \ \text{for all} \ \tau\geq \tau_0$$
		whence it follows that 
		$$\sup_{\zeta\in M}\sup_{w\in L}\sup_{z\in K} \left|S_{\lambda_{\tau}}(f,w,\zeta)(z)-h(w,z)\right|\to 0 \ \text{as} \ \tau\to \infty \ .$$
		From $(2)$, we get $$\sup_{\zeta\in M}\sup_{w\in L}\sup_{z\in M} \left|S_{\lambda_{\tau}}(f,w,\zeta)(z)-f(w,z)\right|<\frac{1}{2\tau} < \frac{1}{\tau} \ \text{for all} \ \tau\geq \tau_0$$
		and so $$\sup_{\zeta\in M}\sup_{w\in L}\sup_{z\in M} \left|S_{\lambda_{\tau}}(f,w,\zeta)(z)-f(w,z)\right| \to 0 \ \text{as} \ \tau\to \infty \ .$$
		Hence, $f\in U^{\mu}(G, \Omega)$ and the proof is completed.
	\end{proof}
	We consider the set $U^\mu(G, \Omega)$ as a subset of the space $H(G \times \Omega)$ endowed with the topology of uniform convergence on all compact subsets of $G \times \Omega$. Since $H(G \times \Omega)$ is a complete metrizable space, Baire's Theorem is at our disposal. So, if we prove that the sets $\bigcup_{n\in \mu} \big[E(\tau, p, m, j, s, n) \bigcap F(\tau, p, s, n)\big]$ are open and dense in $H(G \times \Omega)$ for all positive integers $\tau, p, m, j$ and $s$, then the set $U^\mu(G, \Omega)$ is a dense $G_\delta$ set.
	\begin{lemma}
		\label{Open_1}
		Let $\tau \geq 1, p \geq 1, m \geq 1, j \geq 1, s \geq 1$ and $n \in \mu$. Then, 
		\begin{itemize}
			\item [(i)] the set $E(\tau, p, m, j, s, n)$ is open in the space $H(G \times \Omega)$,
			\item [(ii)] the set $F(\tau, p, s, n)$ is open in the space $H(G \times \Omega)$.
		\end{itemize}
	\end{lemma}
	\begin{proof}
		(i) We shall show that $H(G \times \Omega) \setminus E(\tau, p, m, j, s, n)$ is closed in $H(G \times \Omega)$. Let \break $g_i \in H(G \times \Omega) \setminus E(\tau, p, m, j, s, n), i = 1, 2, \dots$ be a sequence that converges to a function $g \in H(G \times \Omega)$ uniformly on all compact subsets of $G \times \Omega$. We shall show that $g \in H(G \times \Omega) \setminus E(\tau, p, m, j, s, n)$ and so it suffices to show that $$\sup_{\zeta \in M_p}\sup_{w \in F_\tau}\sup_{z \in T_m}\big|S_n(g, w, \zeta)(z) - f_j(w,z)\big| \geq \frac{1}{s}.$$ 
		Let $D$ be a differential operator of mixed partial derivatives in $z = (z_1, \dots, z_d)$, then by the Weierstrass Theorem, we have $D g_i \to D g$ uniformly on all compact subsets of $G \times \Omega$ as $i \to \infty$. So, we have $$\sup_{\zeta \in M_p}\sup_{w \in F_\tau}\big|a_k(g_i, w, \zeta) - a_k(g, w, \zeta) \big| \to 0$$ as $i \to \infty$ for every $0 \leq k \leq n$. Thus, since the set $T_m$ is bounded, we get $$\sup_{\zeta \in M_p}\sup_{w \in F_\tau}\sup_{z \in T_m} \big|S_n(g_i, w, \zeta)(z) - S_n(g, w, \zeta)(z)\big| \to 0$$ as $i \to \infty$ and so $$\sup_{\zeta \in M_p}\sup_{w \in F_\tau}\sup_{z \in T_m}\big|S_n(g_i, w, \zeta)(z) - f_j(w,z)\big| \to \sup_{\zeta \in M_p}\sup_{w \in F_\tau}\sup_{z \in T_m}\big|S_n(g, w, \zeta)(z) - f_j(w,z)\big|$$ as $i \to \infty$. Since $\sup_{\zeta \in M_p}\sup_{w \in F_\tau}\sup_{z \in T_m}\big|S_n(g_i, w, \zeta)(z) - f_j(w,z)\big| \geq \frac{1}{s}$ for all $i = 1, 2, \dots$ we get $\sup_{\zeta \in M_p}\sup_{w \in F_\tau}\sup_{z \in T_m}\big|S_n(g, w, \zeta)(z) - f_j(w,z)\big| \geq \frac{1}{s}$ and the proof is completed.\\
		(ii) Following the previous proof , we can show that the set $H(G \times \Omega) \setminus F(\tau, p, s, n)$ is closed in $H(G \times \Omega)$.
	\end{proof}
	\begin{lemma}
		For every integer $\tau \geq 1, p \geq 1, m \geq 1, j \geq 1$ and $s \geq 1$, the set $\bigcup\limits_{n \in \mu}\big[E(\tau, p, m, j, s, n) \bigcap F(\tau, p, s, n)\big]$ is open and dense in the space $H(G \times \Omega)$.
	\end{lemma}
	\begin{proof}
		By Lemma \ref{Open_1} the sets $E(\tau, p, m, j, s, n) \bigcap F(\tau, p, s, n), n \in \mu$ are open. Therefore the same is true for the union $\bigcup\limits_{n \in \mu}\big[E(\tau, p, m, j, s, n) \bigcap F(\tau, p, s, n)\big]$. We shall prove that this set is also dense. By Lemma \ref{Simply}, it suffices to show that it is dense in the set of polynomials of $r + d$ variables. Let $g$ be a polynomial. Let also $\tilde{F} = \displaystyle\prod_{i = 1}^r \tilde{F}_i, \tilde{M} = \displaystyle\prod_{i = 1}^d \tilde{M}_i$ where $\tilde{F}_i, i = 1, \dots, r$ are compact subsets of $G_i$ with connected complement and $\tilde{M}_i, i = 1, \dots, d$ are compact subsets of $\Omega_i$ with connected complement and $\varepsilon > 0$. We may also assume that $F_\tau \subseteq \tilde{F}$. It suffices to find $n \in \mu$ and $f \in E(\tau, p, m, j, s, n) \bigcap F(\tau, p, s, n)$, such that $$\sup\limits_{w \in \tilde{F}}\sup\limits_{z \in \tilde{M}}|g(w, z) - f(w, z)| < \varepsilon.$$
		The set $T_m$ is of the form $T_m = \displaystyle\prod_{i = 1}^{d} Q_i$, where one of the sets $Q_i$ satisfies $\Omega_i \bigcap Q_i = \emptyset$, which we denote by $Q_{i_0}$. Since $\tilde{M}_{i_0} \subseteq \Omega_{i_0}$, we have $\tilde{M}_{i_0} \bigcap Q_{i_0}  = \emptyset.$ For $i = 1, \dots, d$ with $i \neq i_0$ let $B_i$ be closed balls such that $\tilde{M}_i \bigcup Q_i \subseteq B_i$. We define the function $h : \tilde{F} \times \displaystyle\prod_{i = 1}^d S_i \rightarrow \mathbb{C}$, where $S_i = B_i$ for $i \neq i_0$ and $S_{i_0} = \tilde{M}_{i_0} \bigcup Q_{i_0}$, with: 
		\begin{align*}
		& h(w, z) = g(w, z) \text{ for } w \in \tilde{F} \text{ and } z \text{ in the product of } B_i \text{ for } i \neq i_0 \text{ and } \tilde{M}_{i_0}, \\
		& h(w, z) = f_j(w, z) \text{ for } w \in \tilde{F} \text{ and } z \text{ in the product of } B_i \text{ for } i \neq i_0 \text{ and } Q_{i_0}.
		\end{align*}
		We notice that the functions $g$ and $f_j$ are polynomials and thus defined everywhere. Since $h \in  A_D\big(\tilde{F} \times \displaystyle\prod_{i = 1}^d S_i\big)$ and the set $\tilde{F} \times \displaystyle\prod_{i = 1}^d S_i$ is a product of planar compact sets with connected complement, by Lemma \ref{Merg_1}, there exists a polynomial $p(w, z)$ such that $$\sup\limits_{w \in \tilde{F}} \sup\limits_{z \in \prod_{i = 1}^d S_i}|p(w, z) - h(w, z)| < \min\big\{\varepsilon, \frac{1}{s}\big\}.$$
		By the definitions of the function $h$ and the sets $S_i$ we get $\sup\limits_{w \in \tilde{F}}\sup\limits_{z \in \tilde{M}}|g(w, z) - p(w, z)| < \varepsilon$ and $\sup\limits_{w \in F_\tau}\sup\limits_{z \in T_m}|p(w, z) - f_j(w, z)| < \frac{1}{s}$. For a fixed $\zeta_0 \in M_p$, the polynomial $p(w,z)$ can be written in the form $p(w,z) = \sum\limits_{k = 0}^{\infty}p_k(w)(z - \zeta_0)^{N_k}$ where $p_k$ are polynomials, such that all but finitely many of them are identically equal to 0. For $i = 1, \dots, d$ we set $l_i = \max\{N_k^{(i)} : N_k = (N_k^{(1)}, \dots, N_k^{(d)}), p_k \not\equiv 0, k = 0, 1, \dots \}$. Let $n'$ be an integer such that for each $k = 0, 1, \dots$ with $N_k^{(i)} \leq l_i$ for $i = 1, \dots, d$, we have $k \leq n'$. We notice that $S_n(p, w, \zeta)(z) = p(w, z)$ for all $\zeta \in M_p$ and $n \geq n'$. Thus, by choosing $n \in \mu$ such that $n \geq n'$ we have $$\sup\limits_{\zeta \in M_p}\sup\limits_{w \in F_\tau}\sup\limits_{z \in T_m}|S_n(p, w, \zeta)(z) - f_j(w, z)| < \frac{1}{s}$$
		and
		$$\sup_{\zeta \in M_p}\sup_{w \in F_\tau}\sup_{z \in M_p}\big|S_n(p, w, \zeta)(z) - p(w,z)\big| = 0 < \frac{1}{s}.$$
		This proves that the set $\bigcup\limits_{n \in \mu}\big[E(\tau, p, m, j, s, n) \bigcap F(\tau, p, s, n)\big]$ is indeed dense.
	\end{proof}
	\begin{theorem}
	\label{Main_1}
		Under the above assumptions and notation, the set $U^\mu(G, \Omega)$ is a $G_\delta$ and dense subset of the space $H(G \times \Omega)$ and contains a vector space except $0$, dense in $H(G \times \Omega)$.
	\end{theorem}
	\begin{proof}
		The fact the set $U^\mu(G, \Omega)$ is a $G_\delta$ and dense set is obvious by combining the previous lemmas with Baire's Theorem. The proof that it also contains a dense vector space except $0$ uses the first part of Theorem \ref{Main_1}, follows the lines of the implication $(3)$ implies $(4)$ of the proof of Theorem 3 in \cite{Bayart} and is omitted.
	\end{proof}
	We now consider the case where the center of expansion of the Taylor series of $f$ does not vary but is a fixed point in $\Omega$.
	\begin{definition}
		Let $\zeta_0 \in \Omega, N_k, k = 0, 1, \dots$ be an enumeration of $\mathbb{N}^d$ and $\mu$ be an infinite subset of $\mathbb{N}$. We define $U^\mu(G, \Omega, \zeta_0)$ to be the set of $f \in H(G \times \Omega)$ such that for every set $K = \displaystyle\prod_{i = 1}^{d} K_i$ where $K_i \subseteq \mathbb{C}$ are compact with $\mathbb{C} \setminus K_i$ connected for all $i = 1, \dots d$ and there is at least one $i_0 \in \{1, \dots, d\}$ such that $\Omega_{i_0} \bigcap K_{i_0} = \emptyset$ and every continuous function $h : G \times K \rightarrow \mathbb{C}$ such that for every fixed $w \in G$, the function $K \ni z \rightarrow h(w, z) \in \mathbb{C}$ belongs to $A_D(K)$ and for every fixed $z \in K$, the function $G \ni w \rightarrow h(w, z) \in \mathbb{C}$ belongs to $H(G)$, there exists a strictly increasing sequence $\lambda_n \in \mu, n = 1, 2, \dots$ such that 
		\begin{align*}
		& \sup_{w \in L}\sup_{z \in K}\big|S_{\lambda_n}(f, w, \zeta_0)(z) - h(w,z)\big| \longrightarrow 0 \text{ and } \\
		& \sup_{w \in L}\sup_{z \in M}\big|S_{\lambda_n}(f, w, \zeta_0)(z) - f(w,z)\big| \longrightarrow 0
		\end{align*}
		for every compact subset $L$ of $G$ and for every compact subset of $M$ of $\Omega$.
	\end{definition}
	We notice that the class $U^\mu(G, \Omega, \zeta_0)$ can also be defined without the requirement that the sequence  $(\lambda_n)_{n \in \mathbb{N}}$ is strictly increasing and then the two definitions are equivalent; see \cite{Vlachou}.
	\begin{theorem}
		Under the above assumptions and notation, the set $U^\mu(G, \Omega, \zeta_0)$ is a $G_\delta$ and dense subset of the space $H(G \times \Omega)$ and contains a vector space except $0$, dense in $H(G \times \Omega)$.
	\end{theorem}
	\begin{proof}
		Clearly, we have $U^\mu(G, \Omega) \subseteq U^\mu(G, \Omega, \zeta_0)$, so $U^\mu(G, \Omega, \zeta_0)$ contains a $G_\delta$ and dense set and thus is dense. If we prove that it can be written as a countable intersection of open sets in $H(G \times \Omega)$ then it would be a $G_\delta$ and dense set itself.
		
		Let $F_\tau, M_p, T_m$ and $f_j$ be as previously. For any $\tau, p, m, j, s, n$ with $\tau, p, m, j, s \geq 1, n \geq 0$, we denote by $E(\tau, m, j, s, n)$ and $F(\tau, p, s, n)$ the sets $$E(\tau, m, j, s, n) := \Big\{f \in H(G \times \Omega) : \sup_{w \in F_\tau}\sup_{z \in T_m}\big|S_n(f, w, \zeta_0)(z) - f_j(w,z)\big| < \frac{1}{s}\Big\},$$ \\
		$$F(\tau, p, s, n) := \Big\{f \in H(G \times \Omega) : \sup_{w \in F_\tau}\sup_{z \in M_p}\big|S_n(f, w, \zeta_0)(z) - f(w,z)\big| < \frac{1}{s}\Big\}.$$
		Following the lines of implication of the proofs of Lemma \ref{Inter_1} and Lemma \ref{Open_1} we can show that
		$$U^{\mu}(G, \Omega, \zeta_0)=\bigcap_{\tau, p, m, j, s} \ \ \bigcup_{n\in \mu} \big[E(\tau, m, j, s, n) \bigcap F(\tau, p, s, n)\big].$$
		and that the sets $E(\tau, m, j, s, n)$ and $F(\tau, p, s, n)$ are open in the space $H(G \times \Omega)$ for all $\tau \geq 1, p \geq 1, m \geq 1, j \geq 1, s \geq 1$ and $n \in \mu$. Thus, $U^\mu(G, \Omega, \zeta_0)$ is a $G_\delta$ and dense set. Also, by Theorem \ref{Main_1} and the fact that $U^\mu(G, \Omega) \subseteq U^\mu(G, \Omega, \zeta_0)$, $U^\mu(G, \Omega, \zeta_0)$ contains a vector space except $0$, dense in $H(G \times \Omega)$.
	\end{proof}
	\begin{remark}
	\label{Many}
		We have proven that, for a fixed enumeration of $\mathbb{N}^d$, the set of functions in $H(G \times \Omega)$ whose Taylor series have the desired universal approximation property with respect to this enumeration is a $G_\delta$ and dense set, as described above. Using Baire's Theorem, if we consider the set of functions in $H(G \times \Omega)$ whose Taylor series have the same universal approximation property with respect to any countable family of enumerations of $\mathbb{N}^d$, then this is still a $G_\delta$ and dense set of $H(G \times \Omega)$. Since the set of all enumerations of $\mathbb{N}^d$ is uncountable, a natural question that arises is whether we can generalise the result when the functions in $H(G \times \Omega)$ have the universal approximation property with respect to all enumerations of $\mathbb{N}^d$. The answer to this question is negative. The proof is similar to a result in Section 6 of \cite{Kioulafa}.
	\end{remark}
	\section{Strong Universal Taylor series with parameters in $H(G \times \Omega)$}
	Let $G_i \subseteq \mathbb{C}, i = 1, \dots, r$ and $\Omega_i \subseteq \mathbb{C}, i = 1, \dots, d$ be simply connected domains. We set $G = \displaystyle\prod_{i = 1}^{r} G_i$ and $\Omega = \displaystyle\prod_{i = 1}^{d} \Omega_i$.
	\begin{definition}
		Let $N_k, k = 0, 1, \dots$ be an enumeration of $\mathbb{N}^d$ and $\mu$ be an infinite subset of $\mathbb{N}$. We define $U^\mu_H(G, \Omega)$ to be the set of $f \in H(G \times \Omega)$ such that for every set $K = \displaystyle\prod_{i = 1}^{d} K_i$ where $K_i \subseteq \mathbb{C}$ are compact with $\mathbb{C} \setminus K_i$ connected for all $i = 1, \dots d$ and there is at least one $i_0 \in \{1, \dots, d\}$ such that $\Omega_{i_0} \bigcap K_{i_0} = \emptyset$ and every holomorphic function $h : G \times V \rightarrow \mathbb{C}$ where $V$ is an open set containing $K$, there exists a strictly increasing sequence $\lambda_n \in \mu, n = 1, 2, \dots$ such that 
		\begin{align*}	
		& \sup_{\zeta \in M}\sup_{w \in L}\sup_{z \in K}\big|D \big(S_{\lambda_n}(f, w, \zeta)(z) - h(w,z)\big)\big| \longrightarrow 0 \text{ and } \\
		& \sup_{\zeta \in M}\sup_{w \in L}\sup_{z \in M}\big|S_{\lambda_n}(f, w, \zeta)(z) - f(w,z)\big| \longrightarrow 0
		\end{align*}
		for every compact subset $L$ of $G$, for every compact subset of $M$ of $\Omega$ and every differential operator $D$ of mixed partial derivatives in $(w, z)$.
	\end{definition}
	We notice that the class $U^\mu_H(G, \Omega)$ can also be defined without the requirement that the sequence  $(\lambda_n)_{n \in \mathbb{N}}$ is strictly increasing and then the two definitions are equivalent; see \cite{Vlachou}.
	
	For every positive integer $l$, let $\mathcal{F}_l$ be the set of mixed partial derivatives $D = \dfrac{\partial^{m_1 + \dots + m_{r + d}}}{\partial w_1^{m_1} \cdots \partial w_r^{m_r} \cdot \partial z_1^{m_{r + 1}} \cdots \partial z_d^{m_{r + d}}}$ such that $\sum\limits_{i = 1}^{r + d}m_i \leq l$.
	
	We keep the notation of the sets $M_p, F_\tau$ and $T_m$ and the polynomials $f_j$. For any $\tau, p, m, l, j, s, n$ with $\tau, p, m, l, j, s \geq 1, n \geq 0$, we denote by $E(\tau, p, m, l, j, s, n)$ the set $$\Big\{f \in H(G \times \Omega) : \max_{D \in \mathcal{F}_l}\sup_{\zeta \in M_p}\sup_{w \in F_\tau}\sup_{z \in T_m}\big|D \big(S_n(f, w, \zeta)(z) - f_j(w,z)\big)\big| < \frac{1}{s}\Big\},$$ and by $F(\tau, p, s, n)$ the set 
	$$\Big\{f \in H(G \times \Omega) : \sup_{\zeta \in M_p}\sup_{w \in F_\tau}\sup_{z \in M_p}\big|S_n(f, w, \zeta)(z) - f(w,z)\big| < \frac{1}{s}\Big\}.$$
	\begin{lemma}
		\label{Inter_2}
		With the above assumptions,
		$$U^\mu_H(G, \Omega)=\bigcap_{\tau, p, m, l, j, s} \ \ \bigcup_{n\in \mu} \big[E(\tau, p, m, l, j, s, n) \bigcap F(\tau, p, s, n)\big].$$
	\end{lemma}
	\begin{proof}
		For the inclusion $U^\mu_H(G, \Omega) \subseteq \bigcap\limits_{\tau, p, m, l, j, s} \ \ \bigcup\limits_{n\in \mu} \big[E(\tau, p, m, l, j, s, n) \bigcap F(\tau, p, s, n)\big]$, we consider a function $f\in U^\mu_H(G, \Omega)$ and arbitrary positive integers $\tau, p, m, l, j, s$.
		We shall show that there exists an integer $n' \in \mu$ such that $f \in E(\tau, p, m, l, j, s, n') \bigcap F(\tau, p, s, n')$. Since $f_j$ is a polynomial, it is holomorphic on $G \times \mathbb{C}^d.$ Thus, using the definition of the class $U^\mu_H(G, \Omega)$, we fix a sequence $(\lambda_n)_{n \in \mathbb{N}},\lambda_n\in \mu$, such that 
		$$\sup_{\zeta \in M_p}\sup_{w \in F_\tau}\sup_{z \in T_m}\big|D \big(S_n(f, w, \zeta)(z) - f_j(w,z)\big)\big| \to 0 \ \text{as} \ n \to \infty$$ 
		for every differential operator $D$ of mixed partial derivatives in $(w, z)$ and 
		$$\sup_{\zeta \in M_p}\sup_{w \in F_\tau}\sup_{z \in M_p}\big|S_{\lambda_n}(f, w, \zeta)(z) - f(w,z)\big| \to 0 \ \text{as} \ n \to \infty \ .$$ 
		So, we get 
		$$\max_{D \in \mathcal{F}_l}\sup_{\zeta \in M_p}\sup_{w \in F_\tau}\sup_{z \in T_m}\big|D \big(S_n(f, w, \zeta)(z) - f_j(w,z)\big)\big| \to 0 \ \text{as} \ n \to \infty \ .$$ 
		Thus, there exist positive integers $n_1, n_2$, such that 
		$$\max_{D \in \mathcal{F}_l}\sup_{\zeta \in M_p}\sup_{w \in F_\tau}\sup_{z \in T_m}\big|D \big(S_n(f, w, \zeta)(z) - f_j(w,z)\big)\big|<\frac{1}{s}$$ for every $n \geq n_1$ and $$\sup_{\zeta \in M_p}\sup_{w \in F_\tau}\sup_{z \in M_p}\big|S_{\lambda_n}(f, w, \zeta)(z) - f(w,z)\big| < \frac{1}{s}$$ for every $n \geq n_2$. By setting $n_0 = \max\{n_1, n_2\}$ and $n' = \lambda_{n_0}$ we get $f\in $ \break $E(\tau, p, m, l, j, s, n') \bigcap F(\tau, p, s, n')$ which is exactly what we wanted to show.
		
		For the other inclusion, we consider a function 
		$$f\in \bigcap\limits_{\tau, p, m, l, j, s} \ \ \bigcup\limits_{n\in \mu} \big[E(\tau, p, m, l, j, s, n) \bigcap F(\tau, p, s, n)\big] \ .$$
		Let $K = \displaystyle\prod_{i = 1}^{d} K_i$ where $K_i \subseteq \mathbb{C}$ are compact with $\mathbb{C} \setminus K_i$ connected for all $i = 1, \dots d$ and there is at least one $i_0 \in \{1, \dots, d\}$ such that $\Omega_{i_0} \bigcap K_{i_0} = \emptyset$. Let also  $h : G \times V \rightarrow \mathbb{C}$ be a holomorphic function where $V$ is an open set containing $K$. 
		Then, for each positive integer $\tau$, by Lemma \ref{Merg_2}, there exists a positive integer $j_\tau$, such that 
		$$\sup_{w\in F_{\tau}}\sup_{z\in K}\left|D \big(h(w, z)-f_{j_\tau}(w, z)\big)\right|<\frac{1}{2\tau}$$ 
		for every $D \in \mathcal{F}_\tau$. Since the above relation does not depend on $\zeta$, it can be rewritten as
		$$\max_{D \in \mathcal{F}_\tau}\sup_{\zeta \in M_\tau}\sup_{w\in F_{\tau}}\sup_{z\in K}\left|D \big(h(w, z)-f_{j_\tau}(w, z)\big)\right|<\frac{1}{2\tau} \ . \ (1)$$
		By the definition of the sets $T_m$, there exists a positive integer $m_0$, such that $K \subseteq T_{m_0}$. For each positive integer $\tau$ we have $f \in \bigcup\limits_{n\in \mu} \big[E(\tau, \tau, m_0, \tau, j_\tau, 2 \tau, n) \bigcap F(\tau, \tau, 2 \tau, n)\big]$, thus, there exists a positive integer $\lambda_{\tau}\in \mu$, such that 
		$$\max_{D \in \mathcal{F}_\tau}\sup_{\zeta\in M_{\tau}}\sup_{w\in F_{\tau}}\sup_{z\in T_{m_0}} \left|D \big(S_{\lambda_{\tau}}(f,w,\zeta)(z)-f_{j_\tau}(w,z)\big)\right|<\frac{1}{2\tau}$$ 
		and $$\sup_{\zeta\in M_{\tau}}\sup_{w\in F_{\tau}}\sup_{z\in M_{\tau}} \left|S_{\lambda_{\tau}}(f,w,\zeta)(z)-f(w,z)\right|<\frac{1}{2\tau} \ . \ (2)$$
		Since $K \subset T_{m_0}$, it follows that 
		$$\max_{D \in \mathcal{F}_\tau}\sup_{\zeta\in M_{\tau}}\sup_{w\in F_{\tau}}\sup_{z\in K} \left|D \big(S_{\lambda_{\tau}}(f,w,\zeta)(z)-f_{j_\tau}(w,z)\big)\right|<\frac{1}{2\tau} \ . \ (3)$$
		From $(1)$ and $(3)$ it follows that for each positive integer $\tau$, we have
		$$\max_{D \in \mathcal{F}_\tau}\sup_{\zeta\in M_{\tau}}\sup_{w\in F_{\tau}}\sup_{z\in K} \left|D \big(S_{\lambda_{\tau}}(f,w,\zeta)(z)-h(w,z)\big)\right|<\frac{1}{\tau} \ . \ (4)$$
		
		We consider the sequence $(\lambda_\tau)_{\tau \in \mathbb{N}}$.
		Let $L \subseteq G, M \subseteq \Omega$ be compact sets and $D$ be a differential operator of mixed partial derivatives in $(w, z)$. Then, there exists a positive integer $\tau_0$, such that $L \subseteq F_{\tau_0}, M \subseteq M_{\tau_0}$ and $D \in \mathcal{F}_{\tau_0}$.
		Since the families $\{F_{\tau}:\tau=1,2,\ldots\}, \{M_p:p=1,2,\ldots\}$ and $\{\mathcal{F}_l:l=1,2,\ldots\}$ are increasing, it follows that $L\subseteq F_{\tau}, M\subseteq M_{\tau}$ and $D \in \mathcal{F}_\tau$ for any positive integer $\tau\geq \tau_0$, therefore, from $(4)$, it follows that 
		$$\sup_{\zeta\in M}\sup_{w\in L}\sup_{z\in K} \left|D \big(S_{\lambda_{\tau}}(f,w,\zeta)(z)-h(w,z)\big)\right|<\frac{1}{\tau} \ \text{for all} \ \tau\geq \tau_0$$
		whence it follows that 
		$$\sup_{\zeta\in M}\sup_{w\in L}\sup_{z\in K} \left|D \big(S_{\lambda_{\tau}}(f,w,\zeta)(z)-h(w,z)\big)\right|\to 0 \ \text{as} \ \tau\to \infty \ .$$
		From $(2)$, we get $$\sup_{\zeta\in M}\sup_{w\in L}\sup_{z\in M} \left|S_{\lambda_{\tau}}(f,w,\zeta)(z)-f(w,z)\right|<\frac{1}{2\tau} < \frac{1}{\tau} \ \text{for all} \ \tau\geq \tau_0$$
		and so $$\sup_{\zeta\in M}\sup_{w\in L}\sup_{z\in M} \left|S_{\lambda_{\tau}}(f,w,\zeta)(z)-f(w,z)\right| \to 0 \ \text{as} \ \tau\to \infty \ .$$
		Hence, $f\in U^\mu_H(G, \Omega)$ and the proof is completed.
	\end{proof}
	We consider the set $U^\mu_H(G, \Omega)$ as a subset of the space $H(G \times \Omega)$ endowed with the topology of uniform convergence on all compact subsets of $G \times \Omega$. Since $H(G \times \Omega)$ is a complete metrizable space, Baire's Theorem is at our disposal. So, if we prove that the sets $\bigcup_{n\in \mu} \big[E(\tau, p, m, l, j, s, n) \bigcap F(\tau, p, s, n)\big]$ are open and dense in $H(G \times \Omega)$ for all positive integers $\tau, p, m, l, j$ and $s$, then the set $U^\mu_H(G, \Omega)$ is a dense $G_\delta$ set.
	\begin{lemma}
		\label{Open_2}
		Let $\tau \geq 1, p \geq 1, m \geq 1, l \geq 1, j \geq 1, s \geq 1$ and $n \in \mu$. Then, 
		\begin{itemize}
			\item [(i)] the set $E(\tau, p, m, l, j, s, n)$ is open in the space $H(G \times \Omega)$,
			\item [(ii)] the set $F(\tau, p, s, n)$ is open in the space $H(G \times \Omega)$.
		\end{itemize}
	\end{lemma}
	\begin{proof}
		The proof is similar to the proof of Lemma \ref{Open_1} and is omitted.
	\end{proof}
	\begin{lemma}
		For every integer $\tau \geq 1, p \geq 1, m \geq 1, l \geq 1, j \geq 1$ and $s \geq 1$, the set $\bigcup\limits_{n \in \mu}\big[E(\tau, p, m, l, j, s, n) \bigcap F(\tau, p, s, n)\big]$ is open and dense in the space $H(G \times \Omega)$.
	\end{lemma}
	\begin{proof}
		By Lemma \ref{Open_2} the sets $E(\tau, p, m, l, j, s, n) \bigcap F(\tau, p, s, n), n \in \mu$ are open. Therefore the same is true for the union $\bigcup\limits_{n \in \mu}\big[E(\tau, p, m, l, j, s, n) \bigcap F(\tau, p, s, n)\big]$. We shall prove that this set is also dense. By Lemma \ref{Simply}, it suffices to show that it is dense in the set of polynomials of $r + d$ variables. Let $g$ be a polynomial. Let also $\tilde{F} = \displaystyle\prod_{i = 1}^r \tilde{F}_i, \tilde{M} = \displaystyle\prod_{i = 1}^d \tilde{M}_i$ where $\tilde{F}_i, i = 1, \dots, r$ are compact subsets of $G_i$ with connected complement and $\tilde{M}_i, i = 1, \dots, d$ are compact subsets of $\Omega_i$ with connected complement and $\varepsilon > 0$. We may also assume that $F_\tau \subseteq \tilde{F}$. It suffices to find $n \in \mu$ and $f \in E(\tau, p, m, j, s, n) \bigcap F(\tau, p, s, n)$, such that $$\sup\limits_{w \in \tilde{F}}\sup\limits_{z \in \tilde{M}}|g(w, z) - f(w, z)| < \varepsilon.$$
		The set $T_m$ is of the form $T_m = \displaystyle\prod_{i = 1}^{d} Q_i$, where one of the sets $Q_i$ satisfies $\Omega_i \bigcap Q_i = \emptyset$, which we denote by $Q_{i_0}$. Since $\tilde{M}_{i_0} \subseteq \Omega_{i_0}$, we have $\tilde{M}_{i_0} \bigcap Q_{i_0}  = \emptyset.$ For $i = 1, \dots, d$ with $i \neq i_0$ let $B_i$ be closed balls such that $\tilde{M}_i \bigcup Q_i \subseteq B_i$. We define the function $h : \tilde{F} \times \displaystyle\prod_{i = 1}^d S_i \rightarrow \mathbb{C}$, where $S_i = B_i$ for $i \neq i_0$ and $S_{i_0} = \tilde{M}_{i_0} \bigcup Q_{i_0}$, with: 
		\begin{align*}
		& h(w, z) = g(w, z) \text{ for } w \in \tilde{F} \text{ and } z \text{ in the product of } B_i \text{ for } i \neq i_0 \text{ and } \tilde{M}_{i_0}, \\
		& h(w, z) = f_j(w, z) \text{ for } w \in \tilde{F} \text{ and } z \text{ in the product of } B_i \text{ for } i \neq i_0 \text{ and } Q_{i_0}.
		\end{align*}
		We notice that the functions $g$ and $f_j$ are polynomials and thus defined everywhere. The function $h$ is holomorphic on a neighbourhood of $\tilde{F} \times \displaystyle\prod_{i = 1}^d S_i$ and the set $\tilde{F} \times \displaystyle\prod_{i = 1}^d S_i$ is a product of planar compact sets with connected complement, so, by Lemma \ref{Merg_2}, there exists a polynomial $p(w, z)$ such that
		$$\sup\limits_{w \in \tilde{F}} \sup\limits_{z \in \prod_{i = 1}^d S_i}|D \big(p(w, z) - h(w, z)\big)| < \min\big\{\varepsilon, \frac{1}{s}\big\}.$$ 
		for every $D \in \mathcal{F}_l$. Since the identity operator belongs to $\mathcal{F}_l$, we also get 
		$$\sup\limits_{w \in \tilde{F}} \sup\limits_{z \in \prod_{i = 1}^d S_i}|p(w, z) - h(w, z)| < \min\big\{\varepsilon, \frac{1}{s}\big\}.$$ 
		By the definitions of the function $h$ and the sets $S_i$ we get $\sup\limits_{w \in \tilde{F}}\sup\limits_{z \in \tilde{M}}|p(w,z) - g(w, z)| < \varepsilon$ and $\sup\limits_{w \in F_\tau}\sup\limits_{z \in T_m}|D \big(p(w, z) - f_j(w, z)\big)| < \frac{1}{s}$ for every $D \in \mathcal{F}_l$. For a fixed $\zeta_0 \in M_p$, the polynomial $p(w,z)$ can be written in the form $p(w,z) = \sum\limits_{k = 0}^{\infty}p_k(w)(z - \zeta_0)^{N_k}$ where $p_k$ are polynomials, such that all but finitely many of them are identically equal to $0$. For $i = 1, \dots, d$ we set $l_i = \max\{N_k^{(i)} : N_k = (N_k^{(1)}, \dots, N_k^{(d)}), p_k \not\equiv 0, k = 0, 1, \dots \}$. Let $n'$ be an integer such that for each $k = 0, 1, \dots$ with $N_k^{(i)} \leq l_i$ for $i = 1, \dots, d$, we have $k \leq n'$. We notice that $S_n(p, w, \zeta)(z) = p(w, z)$ for all $\zeta \in M_p$ and $n \geq n'$. Thus, by choosing $n \in \mu$ such that $n \geq n'$ we have $$\sup\limits_{\zeta \in M_p}\sup\limits_{w \in F_\tau}\sup\limits_{z \in T_m}|D \big(S_n(p, w, \zeta)(z) - f_j(w, z) \big)| < \frac{1}{s}$$
		for every $D \in \mathcal{F}_l$ and
		$$\sup\limits_{\zeta \in M_p}\sup\limits_{w \in F_\tau}\sup\limits_{z \in M_p}|S_n(p, w, \zeta)(z) - p(w, z)| = 0 < \frac{1}{s} \ .$$
		This proves that the set $\bigcup\limits_{n \in \mu}\big[E(\tau, p, m, l, j, s, n) \bigcap F(p, l, s, n)\big]$ is indeed dense.
	\end{proof}
	\begin{theorem}
		\label{Main_2}
		Under the above assumptions and notation, the set $U^\mu_H(G, \Omega)$ is a $G_\delta$ and dense subset of the space $H(G \times \Omega)$ and contains a vector space except $0$, dense in $H(G \times \Omega)$.
	\end{theorem}
	\begin{proof}
		The fact the set $U^\mu_H(G, \Omega)$ is a $G_\delta$ and dense set is obvious by combining the previous lemmas with Baire's Theorem. The proof that it also contains a dense vector space except $0$ uses the first part of Theorem \ref{Main_2}, follows the lines of the implication $(3)$ implies $(4)$ of the proof of Theorem 3 in \cite{Bayart} and is omitted.
	\end{proof}
	We now consider the case where the center of expansion of the Taylor series of $f$ does not vary but is a fixed point in $\Omega$.
	\begin{definition}
		Let $\zeta_0 \in \Omega, N_k, k = 0, 1, \dots$ be an enumeration of $\mathbb{N}^d$ and $\mu$ be an infinite subset of $\mathbb{N}$. We define $U^\mu_H(G, \Omega, \zeta_0)$ to be the set of $f \in H(G \times \Omega)$ such that for every set $K = \displaystyle\prod_{i = 1}^{d} K_i$ where $K_i \subseteq \mathbb{C}$ are compact with $\mathbb{C} \setminus K_i$ connected for all $i = 1, \dots d$ and there is at least one $i_0 \in \{1, \dots, d\}$ such that $\Omega_{i_0} \bigcap K_{i_0} = \emptyset$ and every holomorphic function $h : G \times V \rightarrow \mathbb{C}$ where $V$ is an open set containing $K$, there exists a strictly increasing sequence $\lambda_n \in \mu, n = 1, 2, \dots$ such that 
		\begin{align*}	
		& \sup_{w \in L}\sup_{z \in K}\big|D \big(S_{\lambda_n}(f, w, \zeta_0)(z) - h(w,z)\big)\big| \longrightarrow 0 \text{ and } \\
		& \sup_{w \in L}\sup_{z \in M}\big|S_{\lambda_n}(f, w, \zeta_0)(z) - f(w,z)\big| \longrightarrow 0
		\end{align*}
		for every compact subset $L$ of $G$, for every compact subset of $M$ of $\Omega$ and every differential operator $D$ of mixed partial derivatives in $(w, z)$.
	\end{definition}
	We notice that the class $U^\mu_H(G, \Omega, \zeta_0)$ can also be defined without the requirement that the sequence  $(\lambda_n)_{n \in \mathbb{N}}$ is strictly increasing and then the two definitions are equivalent; see \cite{Vlachou}.
	\begin{theorem}
		Under the above assumptions and notation, the set $U^\mu_H(G, \Omega, \zeta_0)$ is a $G_\delta$ and dense subset of the space $H(G \times \Omega)$ and contains a vector space except $0$, dense in $H(G \times \Omega)$.
	\end{theorem}
	\begin{proof}
		Clearly, we have $U^\mu_H(G, \Omega) \subseteq U^\mu_H(G, \Omega, \zeta_0)$, so $U^\mu_H(G, \Omega, \zeta_0)$ contains a $G_\delta$ and dense set and thus is dense. If we prove that it can be written as a countable intersection of open in $H(G \times \Omega)$ then it would be a $G_\delta$ and dense set itself.
		
		Let $F_\tau, M_p, T_m, \mathcal{F}_l$ and $f_j$ be as previously. For any $\tau, p, m, l, j, s, n$ with $\tau, p, m, l, j, s \geq 1, n \geq 0$, we denote by $E(\tau, m, l, j, s, n)$ and $F(\tau, p, s, n)$ the sets $$E(\tau, m, l, j, s, n) := \Big\{f \in H(G \times \Omega) : \max_{D \in \mathcal{F}_l}\sup_{w \in F_\tau}\sup_{z \in T_m}\big|D \big(S_{\lambda_n}(f, w, \zeta_0)(z) - f_j(w,z)\big)\big| < \frac{1}{s}\Big\},$$ \\
		$$F(\tau, p, s, n) := \Big\{f \in H(G \times \Omega) : \sup_{w \in F_\tau}\sup_{z \in M_p}\big|S_n(f, w, \zeta_0)(z) - f(w,z)\big| < \frac{1}{s}\Big\}.$$
		Following the lines of implication of the proofs of Lemma \ref{Inter_2} and Lemma \ref{Open_2} we can show that
		$$U^\mu_H(G, \Omega, \zeta_0)=\bigcap_{\tau, p, m, l, j, s} \ \ \bigcup_{n\in \mu} \big[E(\tau, m, l, j, s, n) \bigcap F(\tau, p, s, n)\big].$$
		and that the sets $E(\tau, m, l, j, s, n)$ and $F(\tau, p, s, n)$ are open in the space $H(G \times \Omega)$ for all $\tau \geq 1, p \geq 1, m \geq 1, l, j \geq 1, s \geq 1$ and $n \in \mu$. Thus, $U^\mu_H(G, \Omega, \zeta_0)$ is a $G_\delta$ and dense set. Also, by Theorem \ref{Main_2} and the fact that $U^\mu_H(G, \Omega) \subseteq U^\mu_H(G, \Omega, \zeta_0)$, $U^\mu_H(G, \Omega, \zeta_0)$ contains a vector space except $0$, dense in $H(G \times \Omega)$.
	\end{proof}
	\begin{remark}
		For Universal Taylor series with parameters with respect to many enumerations, we refer to Remark \ref{Many} and Section 6 of \cite{Kioulafa}.
	\end{remark}
	\section{Universal Taylor series with parameters in \break $A^\infty(G \times \Omega)$}
		Let $G_i \subseteq \mathbb{C}, i = 1, \dots, r$ and $\Omega_i \subseteq \mathbb{C}, i = 1, \dots, d$ be simply connected domains. We set $G = \displaystyle\prod_{i = 1}^{r} G_i$ and $\Omega = \displaystyle\prod_{i = 1}^{d} \Omega_i$. We denote by $A^\infty(G \times \Omega)$ the set of all holomorphic function on $G \times \Omega$ such that the function $Df$ extends continuously to $\overline{G \times \Omega}$ for all differential operators $D$ of mixed partial derivatives in $(w, z) = (w_1, \dots, w_r, z_1, \dots, z_d)$. The topology of $A^\infty(G \times \Omega)$ is defined by the countable family of seminorms $\sup\limits_{w \in \overline{G}, |w| \leq n}\sup\limits_{z \in \overline{\Omega}, |z| \leq n}\big|Df(w, z)\big|$ where $n = 1, 2, \dots$ and $D$ varies in the set of all differential operators of mixed partial derivatives in $(w, z) = (w_1, \dots, w_r, z_1, \dots, z_d)$. We notice that for $w \in \partial G$ and $z \in \partial \Omega$, $Df(w, z)$ is well defined since $Df$ extends continuously to $\overline{G \times \Omega}$. The space $A^\infty(G \times \Omega)$ endowed with this topology becomes a complete metrizable space and so Baire's Theorem is at our disposal. Let also $X^\infty(G \times \Omega)$ be the closure of the set of all polynomials in $A^\infty(G \times \Omega)$.
	\begin{definition}
		Let $N_k, k = 0, 1, \dots$ be an enumeration of $\mathbb{N}^d, \mu$ be an infinite subset of $\mathbb{N}$ and $G$ and $\Omega$ be as above. We also assume that the sets $\{\infty\} \bigcup [\mathbb{C} \setminus \overline{G}_i], i = 1, \dots, r$ and $\{\infty\} \bigcup [\mathbb{C} \setminus \overline{\Omega}_i], i = 1, \dots, d$ are connected. We define $U^\mu_\infty(G, \Omega)$ to be the set of $f \in X^\infty(G \times \Omega)$ such that for every set $K = \displaystyle\prod_{i = 1}^{d} K_i$ where $K_i \subseteq \mathbb{C}$ are compact with $\mathbb{C} \setminus K_i$ connected for all $i = 1, \dots d$ and there is at least one $i_0 \in \{1, \dots, d\}$ such that $\overline{\Omega}_{i_0} \bigcap K_{i_0} = \emptyset$ and every continuous function $h : G \times K \rightarrow \mathbb{C}$ such that for every fixed $w \in G$, the function $K \ni z \rightarrow h(w, z) \in \mathbb{C}$ belongs to $A_D(K)$ and for every fixed $z \in K$, the function $G \ni w \rightarrow h(w, z) \in \mathbb{C}$ belongs to $H(G)$, there exists a strictly increasing sequence $\lambda_n \in \mu, n = 1, 2, \dots$ such that 
		\begin{align*}
		& \sup_{\zeta \in M}\sup_{w \in L}\sup_{z \in K}\big|S_{\lambda_n}(f, w, \zeta)(z) - h(w,z)\big| \longrightarrow 0 \text{ and } \\
		& \sup_{\zeta \in M}\sup_{w \in \overline{G}, |w| \leq l}\sup_{z \in \overline{\Omega}, |z| \leq l}\big|D \big(S_{\lambda_n}(f, w, \zeta)(z) - f(w,z)\big)\big| \longrightarrow 0
		\end{align*}
		for every compact subset $L$ of $G$, every compact subset $M$ of $\overline{\Omega}$, every differential operator $D$ of mixed partial derivatives in $(w, z)$ and every positive integer $l$.
	\end{definition}
	We notice that the class $U^\mu_\infty(G, \Omega)$ can also be defined without the requirement that the sequence  $(\lambda_n)_{n \in \mathbb{N}}$ is strictly increasing and then the two definitions are equivalent; see \cite{Vlachou}.
	
	Let $M_p = \{z \in \overline{\Omega} : |z| \leq  p \}, p = 1, 2, \dots$. It is known that, for each $i = 1, \dots, d$, there exist compact sets $R_{i, j} \subseteq \mathbb{C} \setminus \overline{\Omega}_i, j = 1, 2, \dots$ with connected complement, such that for every compact set $T \subseteq \mathbb{C} \setminus \overline{\Omega}_i$ with connected complement, there exists an integer $j$ such that $T \subseteq R_{i, j}$. Let $T_m$ be an enumeration of all $\displaystyle\prod_{i = 1}^{d} Q_i$, such that there exists an integer $i_0 = 1, 2, \dots, d$ with $Q_{i_0} \in \{R_{i_0, j} : j = 1, 2, \dots,\}$ and the rest of the sets $Q_i$ are closed balls centered at 0 whose radius is a positive integer.
	
	We keep the notation of the sets $F_\tau$ and $\mathcal{F}_l$ and the polynomials $f_j$ from the previous section. For any $\tau, p, m, l, j, s, n$ with $\tau, p, m, l, j, s \geq 1, n \geq 0$, we denote by $E(\tau, p, m, j, s, n)$ the set $$\Big\{f \in X^\infty(G \times \Omega) : \sup_{\zeta \in M_p}\sup_{w \in F_\tau}\sup_{z \in T_m}\big|S_n(f, w, \zeta)(z) - f_j(w,z)\big| < \frac{1}{s}\Big\},$$ and by $F(p, l, s, n)$ the set
	$$\Big\{f \in X^\infty(G \times \Omega) : \max_{D \in \mathcal{F}_l}\sup_{\zeta \in M_p}\sup_{\substack{w \in \overline{G} \\ |w| \leq l}}\sup_{\substack{z \in \overline{\Omega} \\ |z| \leq l}}\big|D \big(S_n(f, w, \zeta)(z) - f(w,z)\big)\big| < \frac{1}{s}\Big\}.$$
	\begin{lemma}
		\label{Inter_3}
		With the above assumptions,
		$$U^\mu_\infty(G, \Omega)=\bigcap_{\tau, p, m, l, j, s} \ \ \bigcup_{n\in \mu} \big[E(\tau, p, m, j, s, n) \bigcap F(p, l, s, n)\big].$$
	\end{lemma}
	\begin{proof}
		For the inclusion $U^\mu_\infty(G, \Omega) \subseteq \bigcap\limits_{\tau, p, m, l, j, s} \ \ \bigcup\limits_{n\in \mu} \big[E(\tau, p, m, j, s, n) \bigcap F(p, l, s, n)\big]$, we consider a function $f\in U^\mu_\infty(G, \Omega)$ and arbitrary positive integers $\tau, p, m, l, j, s$.
		We shall show that there exists an integer $n' \in \mu$ such that $f \in E(\tau, p, m, j, s, n') \bigcap F(p, l, s, n')$. Since $f_j$ is a polynomial, it is holomorphic on $\mathbb{C}^r \times \mathbb{C}^d.$ Thus, using the definition of the class $U^\mu_\infty(G, \Omega)$, we fix a sequence $(\lambda_n)_{n \in \mathbb{N}},\lambda_n\in \mu$, such that 
		$$\sup_{\zeta \in M_p}\sup_{w \in F_\tau}\sup_{z \in T_m}\big|S_{\lambda_n}(f, w, \zeta)(z) - f_j(w,z)\big| \to 0 \ \text{as} \ n \to \infty$$ 
		and 
		$$\sup_{\zeta \in M_p}\sup_{w \in \overline{G}, |w| \leq l}\sup_{z \in \overline{\Omega}, |z| \leq l}\big|D \big(S_{\lambda_n}(f, w, \zeta)(z) - f(w,z)\big)\big| \to 0 \ \text{as} \ n \to \infty \ .$$ 
		for every differential operator $D$ of mixed partial derivatives in $(w, z)$ and so
		$$\max_{D \in \mathcal{F}_l}\sup_{\zeta \in M_p}\sup_{w \in \overline{G}, |w| \leq l}\sup_{z \in \overline{\Omega}, |z| \leq l}\big|D \big(S_{\lambda_n}(f, w, \zeta)(z) - f(w,z)\big)\big| \to 0 \ \text{as} \ n \to \infty \ .$$ 
		Thus, there exist positive integers $n_1, n_2$, such that 
		$$\sup_{\zeta \in M_p}\sup_{w \in F_\tau}\sup_{z \in T_m}\big|S_{\lambda_n}(f, w, \zeta)(z) - f_j(w,z)\big|<\frac{1}{s}$$ for every $n \geq n_1$ and $$\max_{D \in \mathcal{F}_l}\sup_{\zeta \in M_p}\sup_{w \in \overline{G}, |w| \leq l}\sup_{z \in \overline{\Omega}, |z| \leq l}\big|D \big(S_{\lambda_n}(f, w, \zeta)(z) - f(w,z)\big)\big| < \frac{1}{s}$$ for every $n \geq n_2$. By setting $n_0 = \max\{n_1, n_2\}$ and $n' = \lambda_{n_0}$ we get $f\in $ \break $E(\tau, p, m, j, s, n') \bigcap F(p, l, s, n')$ which is exactly what we wanted to show.
		
		For the other inclusion, we consider a function 
		$$f\in \bigcap\limits_{\tau, p, m, j, s} \ \ \bigcup\limits_{n\in \mu} \big[E(\tau, p, m, j, s, n) \bigcap F(p, l, s, n)\big] \ .$$
		Let $K = \displaystyle\prod_{i = 1}^{d} K_i$ where $K_i \subseteq \mathbb{C}$ are compact with $\mathbb{C} \setminus K_i$ connected for all $i = 1, \dots d$ and there is at least one $i_0 \in \{1, \dots, d\}$ such that $\overline{\Omega}_{i_0} \bigcap K_{i_0} = \emptyset$. Let also $h : G \times K \rightarrow \mathbb{C}$ be a continuous function such that for every fixed $w \in G$, the function $K \ni z \rightarrow h(w, z) \in \mathbb{C}$ belongs to $A_D(K)$ and for every fixed $z \in K$, the function $G \ni w \rightarrow h(w, z) \in \mathbb{C}$ belongs to $H(G)$. 
		Then, for each positive integer $\tau$ we have $h\in A_D(F_{\tau} \times K)$; thus, by Lemma \ref{Merg_1}, there exists a positive integer $j_\tau$, such that 
		$$\sup_{w\in F_{\tau}}\sup_{z\in K}\left|h(w, z)-f_{j_\tau}(w, z)\right|<\frac{1}{2\tau} \ .$$
		Since the above relation does not depend on $\zeta$, it can be rewritten as
		$$\sup_{\zeta \in M_\tau}\sup_{w\in F_{\tau}}\sup_{z\in K}\left|h(w, z)-f_{j_\tau}(w, z)\right|<\frac{1}{2\tau} \ . \ (1)$$
		By the definition of the sets $T_m$, there exists a positive integer $m_0$, such that $K \subseteq T_{m_0}$. For each positive integer $\tau$ we have $f \in \bigcup\limits_{n\in \mu} \big[E(\tau, \tau, m_0, j_\tau, 2 \tau, n) \bigcap F(\tau, \tau, 2 \tau, n)\big]$, thus, there exists a positive integer $\lambda_{\tau}\in \mu$, such that 
		$$\sup_{\zeta\in M_{\tau}}\sup_{w\in F_{\tau}}\sup_{z\in T_{m_0}} \left|S_{\lambda_{\tau}}(f,w,\zeta)(z)-f_{j_\tau}(w,z)\right|<\frac{1}{2\tau}$$ 
		and $$\max_{D \in \mathcal{F}_\tau}\sup_{\zeta \in M_\tau}\sup_{w \in \overline{G}, |w| \leq \tau}\sup_{z \in \overline{\Omega}, |z| \leq \tau}\big|D \big(S_{\lambda_\tau}(f, w, \zeta)(z) - f(w,z)\big)\big|<\frac{1}{2\tau} \ . \ (2)$$
		Since $K \subset T_{m_0}$, it follows that 
		$$\sup_{\zeta\in M_{\tau}}\sup_{w\in F_{\tau}}\sup_{z\in K} \left|S_{\lambda_{\tau}}(f,w,\zeta)(z)-f_{j_\tau}(w,z)\right|<\frac{1}{2\tau} \ . \ (3)$$
		From $(1)$ and $(3)$ it follows that for each positive integer $\tau$, we have
		$$\sup_{\zeta\in M_{\tau}}\sup_{w\in F_{\tau}}\sup_{z\in K} \left|S_{\lambda_{\tau}}(f,w,\zeta)(z)-h(w,z)\right|<\frac{1}{\tau} \ . \ (4)$$
		
		We consider the sequence $(\lambda_\tau)_{\tau \in \mathbb{N}}$.
		Let $L \subseteq G$, $M \subseteq \overline{\Omega}$ be compact sets, $D$ be a differential operator of mixed partial derivatives in $(w, z)$ and $l$ a positive integer.
		Then, there exists a positive integer $\tau_0 \geq l$, such that $L \subseteq F_{\tau_0}$, $M \subseteq M_{\tau_0}$ and $D \in \mathcal{F}_{\tau_0}$.
		Since the families $\{F_{\tau}:\tau=1,2,\ldots\}$, $\{M_p:p=1,2,\ldots\}$ and $\{\mathcal{F}_l:l=1,2,\ldots\}$ are increasing, it follows that $L\subseteq F_{\tau}$, $M\subseteq M_{\tau}$ and $D \in \mathcal{F}_{\tau}$ for any positive integer $\tau\geq \tau_0$, therefore, from $(4)$, it follows that 
		$$\sup_{\zeta\in M}\sup_{w\in L}\sup_{z\in K} \left|S_{\lambda_{\tau}}(f,w,\zeta)(z)-h(w,z)\right|<\frac{1}{\tau} \ \text{for all} \ \tau\geq \tau_0$$
		whence it follows that 
		$$\sup_{\zeta\in M}\sup_{w\in L}\sup_{z\in K} \left|S_{\lambda_{\tau}}(f,w,\zeta)(z)-h(w,z)\right|\to 0 \ \text{as} \ \tau\to \infty \ .$$
		From $(2)$, we get $$\sup_{\zeta \in M}\sup_{w \in \overline{G}, |w| \leq l}\sup_{z \in \overline{\Omega}, |z| \leq l}\big|D \big(S_{\lambda_\tau}(f, w, \zeta)(z) - f(w,z)\big)\big|<\frac{1}{2\tau} < \frac{1}{\tau} \ \text{for all} \ \tau\geq \tau_0$$
		and so $$\sup_{\zeta \in M}\sup_{w \in \overline{G}, |w| \leq l}\sup_{z \in \overline{\Omega}, |z| \leq l}\big|D \big(S_{\lambda_\tau}(f, w, \zeta)(z) - f(w,z)\big)\big| \to 0 \ \text{as} \ \tau\to \infty \ .$$
		Hence, $f\in U^\mu_\infty(G, \Omega)$ and the proof is completed.
	\end{proof}
	The space $X^\infty(G \times \Omega)$ is a closed subset of $A^\infty(G \times \Omega)$ so it is a complete metrizable space. Thus, Baire's Theorem is at our disposal. So, if we prove that the sets $\bigcup_{n\in \mu} \big[E(\tau, p, m, j, s, n) \bigcap F(p, l, s, n)\big]$ are open and dense in $X^\infty(G \times \Omega)$ for all positive integers $\tau, p, m, l, j$ and $s$, then the set $U^\mu_\infty(G, \Omega)$ is a dense $G_\delta$ set.
	\begin{lemma}
		\label{Open_3}
		Let $\tau \geq 1, p \geq 1, m \geq 1, l \geq 1, j \geq 1, s \geq 1$ and $n \in \mu$. Then:
		\begin{itemize}
			\item [(i)] The set $E(\tau, p, m, j, s, n)$ is open in the space $X^\infty(G \times \Omega)$.
			\item [(ii)] The set $F(p, l, s, n)$ is open in the space $X^\infty(G \times \Omega)$.
		\end{itemize}
	\end{lemma}
	\begin{proof}
		(i) We shall show that $X^\infty(G \times \Omega) \setminus E(\tau, p, m, j, s, n)$ is closed in $H(G \times \Omega)$. Let $g_i \in X^\infty(G \times \Omega) \setminus E(\tau, p, m, j, s, n), i = 1, 2, \dots$ be a sequence that converges to a function $g \in X^\infty(G \times \Omega)$ in the topology of $A^\infty(G \times \Omega)$. We shall show that $g \in X^\infty(G \times \Omega) \setminus E(\tau, p, m, j, s, n)$. Since $g_i \to g$ in the topology of $A^\infty(G \times \Omega)$, we notice that $D g_i \to D g$ uniformly on all compact subsets of $\overline{G \times \Omega}$ for all differential operators $D$ of mixed partial derivatives in $(w, z)$ and the result follows as in Lemma \ref{Open_1} (i).\\
		(ii) We shall show that $X^\infty(G \times \Omega) \setminus F(p, l, s, n)$ is closed in $X^\infty(G \times \Omega)$. Let $g_i \in X^\infty(G \times \Omega) \setminus F(p, l, s, n), i = 1, 2, \dots$ be a sequence that converges to a function $g \in X^\infty(G \times \Omega)$ in the topology of $A^\infty(G \times \Omega)$. We shall show that $g \in X^\infty(G \times \Omega) \setminus F(p, l, s, n)$ and so it suffices to show that $$\max_{D \in \mathcal{F}_l}\sup_{\zeta \in M_p}\sup_{\substack{w \in \overline{G} \\ |w| \leq l}}\sup_{\substack{z \in \overline{\Omega} \\ |z| \leq l}}\big|D \big(S_n(g, w, \zeta)(z) - g(w,z)\big)\big| \geq \frac{1}{s}.$$ 
		Let $D \in \mathcal{F}_l$, then, since $g_i \to g$ in the topology of $A^\infty(G \times \Omega)$, we have $D g_i \to D g$ uniformly on all compact subsets of $\overline{G \times \Omega}$ as $i \to \infty$. So, we get $$\sup_{\zeta \in M_p}\sup_{\substack{w \in \overline{G} \\ |w| \leq l}}\sup_{\substack{z \in \overline{\Omega} \\ |z| \leq l}} \big|D \big(S_n(g_i, w, \zeta)(z) - S_n(g, w, \zeta)(z)\big)\big| \to 0$$ as $i \to \infty$ and so $$\sup_{\zeta \in M_p}\sup_{\substack{w \in \overline{G} \\ |w| \leq l}}\sup_{\substack{z \in \overline{\Omega} \\ |z| \leq l}}\big|D \big(S_n(g_i, w, \zeta)(z) - g_i(w,z)\big)\big|\to \sup_{\zeta \in M_p}\sup_{\substack{w \in \overline{G} \\ |w| \leq l}}\sup_{\substack{z \in \overline{\Omega} \\ |z| \leq l}}\big|D \big(S_n(g, w, \zeta)(z) - g(w,z)\big)\big|$$ as $i \to \infty$. Since $\max\limits_{D \in \mathcal{F}_l}\sup\limits_{\zeta \in M_p}\sup\limits_{\substack{w \in \overline{G} \\ |w| \leq l}}\sup\limits_{\substack{z \in \overline{\Omega} \\ |z| \leq l}}\big|D \big(S_n(g_i, w, \zeta)(z) - g_i(w,z)\big)\big| \geq \frac{1}{s}$ for all $i = 1, 2, \dots$ we get $\max\limits_{D \in \mathcal{F}_l}\sup\limits_{\zeta \in M_p}\sup\limits_{\substack{w \in \overline{G} \\ |w| \leq l}}\sup\limits_{\substack{z \in \overline{\Omega} \\ |z| \leq l}}\big|D \big(S_n(g, w, \zeta)(z) - g(w,z)\big)\big| \geq \frac{1}{s}$ and the proof is completed.
	\end{proof}
	\begin{lemma}
		For every integer $\tau \geq 1, p \geq 1, m \geq 1, l \geq 1, j \geq 1$ and $s \geq 1$, the set $\bigcup\limits_{n \in \mu}\big[E(\tau, p, m, j, s, n) \bigcap F(p, l, s, n)\big]$ is open and dense in the space $X^\infty(G \times \Omega)$.
	\end{lemma}
	\begin{proof}
		By Lemma \ref{Open_3} the sets $E(\tau, p, m, j, s, n) \bigcap F(p, l, s, n), n \in \mu$ are open. Therefore the same is true for the union $\bigcup\limits_{n \in \mu}\big[E(\tau, p, m, j, s, n) \bigcap F(p, l, s, n)\big]$. We shall prove that this set is also dense. By the definition of the set $X^\infty(G \times \Omega)$, it suffices to show that $\bigcup\limits_{n \in \mu}\big[E(\tau, p, m, j, s, n) \bigcap F(p, l, s, n)\big]$ is dense in the set of polynomials. Let $g$ be a polynomial. Let also $\mathcal{F}$ be a finite set of mixed partial derivatives in $(w, z), l_0$ a positive integer and $\varepsilon > 0$. It suffices to find $n \in \mu$ and $f \in E(\tau, p, m, j, s, n) \bigcap F(p, l, s, n)$, such that $$\sup_{\substack{w \in \overline{G} \\ |w| \leq l_0}}\sup_{\substack{z \in \overline{\Omega} \\ |z| \leq l_0}}|D \big(f(w,z) - g(w, z)\big)| < \varepsilon.$$ for every $D \in \mathcal{F}$.
		Let also $\tilde{F} = \displaystyle\prod_{i = 1}^r \tilde{F}_i, \tilde{M} = \displaystyle\prod_{i = 1}^d \tilde{M}_i$ where $\tilde{F}_i, i = 1, \dots, r$ are compact subsets of $\overline{G}_i$ with connected complement and $\tilde{M}_i, i = 1, \dots, d$ are compact subsets of $\overline{\Omega}_i$ with connected complement. Since the sets $\{\infty\} \bigcup [\mathbb{C} \setminus \overline{G}_i], i = 1, \dots, r$ and \break $\{\infty\} \bigcup [\mathbb{C} \setminus \overline{\Omega}_i], i = 1, \dots, d$ are connected we get that the sets $\overline{G}_i \bigcap \{w_i \in \mathbb{C} : |w_i| \leq l_0\}, i = 1, \dots, r$ and $\overline{\Omega}_i \bigcap \{w_i \in \mathbb{C} : |w_i| \leq l_0\}, i = 1, \dots, d$ have connected complement. So, we can assume that $\big(\overline{G} \bigcap \{w \in \mathbb{C}^r : |w| \leq l_0\}\big) \bigcup F_\tau \subseteq \tilde{F}$ and $\big(\overline{\Omega} \bigcap \{z \in \mathbb{C}^d : |z| \leq l_0\}\big) \subseteq \tilde{M}$.
		The set $T_m$ is of the form $T_m = \displaystyle\prod_{i = 1}^{d} Q_i$, where one of the sets $Q_i$ satisfies $\overline{\Omega}_i \bigcap Q_i = \emptyset$, which we denote by $Q_{i_0}$. Since $\tilde{M}_i \subseteq \overline{\Omega}_{i_0}$ we get $\tilde{M}_i \bigcap Q_i = \emptyset$ For $i = 1, \dots, d$ with $i \neq i_0$ let $B_i$ be closed balls such that $\tilde{M}_i \bigcup Q_i \subseteq B_i$. We define the function $h : \tilde{F} \times \displaystyle\prod_{i = 1}^d S_i \rightarrow \mathbb{C}$, where $S_i = B_i$ for $i \neq i_0$ and $S_{i_0} = \tilde{M}_{i_0} \bigcup Q_{i_0}$, with: 
		\begin{align*}
		& h(w, z) = g(w, z) \text{ for } w \in \tilde{F} \text{ and } z \text{ in the product of } B_i \text{ for } i \neq i_0 \text{ and } \tilde{M}_{i_0}, \\
		& h(w, z) = f_j(w, z) \text{ for } w \in \tilde{F} \text{ and } z \text{ in the product of } B_i \text{ for } i \neq i_0 \text{ and } Q_{i_0}.
		\end{align*}
		We notice that the functions $g$ and $f_j$ are polynomials and thus defined everywhere. The function $h$ is holomorphic on a neighbourhood of $\tilde{F} \times \displaystyle\prod_{i = 1}^d S_i$ and the set $\tilde{F} \times \displaystyle\prod_{i = 1}^d S_i$ is a product of planar compact sets with connected complement, so, by Lemma \ref{Merg_2}, there exists a polynomial $p(w, z)$ such that $$\sup\limits_{w \in \tilde{F}} \sup\limits_{z \in \prod_{i = 1}^d S_i}|p(w, z) - h(w, z)| < \min\big\{\varepsilon, \frac{1}{s}\big\}.$$
		and
		$$\sup\limits_{w \in \tilde{F}} \sup\limits_{z \in \prod_{i = 1}^d S_i}|D \big(p(w, z) - h(w, z)\big)| < \min\big\{\varepsilon, \frac{1}{s}\big\}.$$ for every $D \in \mathcal{F}$.
		By the definitions of the function $h$ and the sets $S_i$ we get \break $\sup\limits_{w \in \overline{G}, |w| \leq l_0}\sup\limits_{z \in \overline{\Omega}, |z| \leq l_0}|D \big(p(w,z) - g(w, z)\big)| < \varepsilon$ for every $D \in \mathcal{F}$ and $\sup\limits_{w \in F_\tau}\sup\limits_{z \in T_m}|p(w, z) - f_j(w, z)| < \frac{1}{s}$. For a fixed $\zeta_0 \in M_p$, the polynomial $p(w,z)$ can be written in the form $p(w,z) = \sum\limits_{k = 0}^{\infty}p_k(w)(z - \zeta_0)^{N_k}$ where $p_k$ are polynomials, such that all but finitely many of them are identically equal to 0. For $i = 1, \dots, d$ we set $l_i = \max\{N_k^{(i)} : N_k = (N_k^{(1)}, \dots, N_k^{(d)}), p_k \not\equiv 0, k = 0, 1, \dots \}$. Let $n'$ be an integer such that for each $k = 0, 1, \dots$ with $N_k^{(i)} \leq l_i$ for $i = 1, \dots, d$, we have $k \leq n'$. We notice that $S_n(p, w, \zeta)(z) = p(w, z)$ for all $\zeta \in M_p$ and $n \geq n'$. Thus, by choosing $n \in \mu$ such that $n \geq n'$ we have $$\sup\limits_{\zeta \in M_p}\sup\limits_{w \in F_\tau}\sup\limits_{z \in T_m}|S_n(p, w, \zeta)(z) - f_j(w, z)| < \frac{1}{s}$$
		and
		$$\max_{D \in \mathcal{F}_l}\sup_{\zeta \in M_p}\sup_{\substack{w \in \overline{G} \\ |w| \leq l}}\sup_{\substack{z \in \overline{\Omega} \\ |z| \leq l}}\big|D \big(S_n(p, w, \zeta)(z) - p(w,z)\big)\big| = 0 < \frac{1}{s} \ .$$
		This proves that the set $\bigcup\limits_{n \in \mu}\big[E(\tau, p, m, j, s, n) \bigcap F(p, l, s, n)\big]$ is indeed dense.
	\end{proof}
	\begin{theorem}
		\label{Main_3}
		Under the above assumptions and notation, the set $U^\mu_\infty(G, \Omega)$ is a $G_\delta$ and dense subset of the space $X^\infty(G \times \Omega)$ and contains a vector space except $0$, dense in $X^\infty(G \times \Omega)$.
	\end{theorem}
	\begin{proof}
		The fact the set $U^\mu_\infty(G, \Omega)$ is a $G_\delta$ and dense set is obvious by combining the previous lemmas with Baire's Theorem. The proof that it also contains a dense vector space except $0$ uses the first part of Theorem \ref{Main_3}, follows the lines of the implication $(3)$ implies $(4)$ of the proof of Theorem 3 in \cite{Bayart} and is omitted.
	\end{proof}
	We now consider the case where the center of expansion of the Taylor series of $f$ does not vary but is a fixed point in $\overline{\Omega}$.
	\begin{definition}
		Let $\zeta_0 \in \overline{\Omega}, N_k, k = 0, 1, \dots$ be an enumeration of $\mathbb{N}^d$ and $\mu$ be an infinite subset of $\mathbb{N}$. We define $U^\mu_\infty(G, \Omega, \zeta_0)$ to be the set of $f \in X^\infty(G \times \Omega)$ such that for every set $K = \displaystyle\prod_{i = 1}^{d} K_i$ where $K_i \subseteq \mathbb{C}$ are compact with $\mathbb{C} \setminus K_i$ connected for all $i = 1, \dots d$ and there is at least one $i_0 \in \{1, \dots, d\}$ such that $\overline{\Omega}_{i_0} \bigcap K_{i_0} = \emptyset$ and every continuous function $h : G \times K \rightarrow \mathbb{C}$ such that for every fixed $w \in G$, the function $K \ni z \rightarrow h(w, z) \in \mathbb{C}$ belongs to $A_D(K)$ and for every fixed $z \in K$, the function $G \ni w \rightarrow h(w, z) \in \mathbb{C}$ belongs to $H(G)$, there exists a strictly increasing sequence $\lambda_n \in \mu, n = 1, 2, \dots$ such that 
		\begin{align*}
		& \sup_{w \in L}\sup_{z \in K}\big|S_{\lambda_n}(f, w, \zeta_0)(z) - h(w,z)\big| \longrightarrow 0 \text{ and } \\
		& \sup_{w \in \overline{G}, |w| \leq l}\sup_{z \in \overline{\Omega}, |z| \leq l}\big|D \big(S_{\lambda_n}(f, w, \zeta_0)(z) - f(w,z)\big)\big| \longrightarrow 0
		\end{align*}
		for every compact subset $L$ of $G$, every differential operator $D$ of mixed partial derivatives in $(w, z)$ and every positive integer $l$.
	\end{definition}
	We notice that the class $U^\mu_\infty(G, \Omega, \zeta_0)$ can also be defined without the requirement that the sequence  $(\lambda_n)_{n \in \mathbb{N}}$ is strictly increasing and then the two definitions are equivalent; see \cite{Vlachou}.
	\begin{theorem}
		Under the above assumptions and notation, the set $U^\mu_\infty(G, \Omega, \zeta_0)$ is a $G_\delta$ and dense subset of the space $X^\infty(G \times \Omega)$ and contains a vector space except $0$, dense in $X^\infty(G \times \Omega)$.
	\end{theorem}
	\begin{proof}
		Clearly, we have $U^\mu_\infty(G, \Omega) \subseteq U^\mu_\infty(G, \Omega, \zeta_0)$, so $U^\mu_\infty(G, \Omega, \zeta_0)$ contains a $G_\delta$ and dense set and thus is dense. If we prove that it can be written as a countable intersection of open in $X^\infty(G \times \Omega)$ then it would be a $G_\delta$ and dense set itself.
		
		Let $F_\tau, T_m, \mathcal{F}_l$ and $f_j$ be as previously. For any $\tau, m, l, j, s, n$ with $\tau, m, l, j, s \geq 1, n \geq 0$, we denote by $E(\tau, m, j, s, n)$ and $F(l, s, n)$ the sets $$E(\tau, m, j, s, n) := \Big\{f \in X^\infty(G \times \Omega) : \sup_{w \in F_\tau}\sup_{z \in T_m}\big|S_n(f, w, \zeta_0)(z) - f_j(w,z)\big| < \frac{1}{s}\Big\},$$ \\
		$$F(l, s, n) := \Big\{f \in X^\infty(G \times \Omega) : \max_{D \in \mathcal{F}_l}\sup_{\substack{w \in \overline{G} \\ |w| \leq l}}\sup_{\substack{z \in \overline{\Omega} \\ |z| \leq l}}\big|D \big(S_n(f, w, \zeta_0)(z) - f(w,z)\big)\big| < \frac{1}{s}\Big\}.$$
		Following the lines of implication of the proofs of Lemma \ref{Inter_3} and Lemma \ref{Open_3} we can show that
		$$U^\mu_\infty(G, \Omega, \zeta_0)=\bigcap_{\tau, m, l, j, s} \ \ \bigcup_{n\in \mu} \big[E(\tau, m, j, s, n) \bigcap F(l, s, n)\big].$$
		and that the sets $E(\tau, m, j, s, n)$ and $F(l, s, n)$ are open in the space $X^\infty(G \times \Omega)$ for all $\tau \geq 1,, m \geq 1, l \geq 1, j \geq 1, s \geq 1$ and $n \in \mu$. Thus, $U^\mu_\infty(G, \Omega, \zeta_0)$ is a $G_\delta$ and dense set. Also, by Theorem \ref{Main_3} and the fact that $U^\mu_\infty(G, \Omega) \subseteq U^\mu_\infty(G, \Omega, \zeta_0)$, $U^\mu_\infty(G, \Omega, \zeta_0)$ contains a vector space except $0$, dense in $X^\infty(G \times \Omega)$.
	\end{proof}
	\begin{remark}
		For Universal Taylor series with parameters in $X^\infty(G \times \Omega)$ with respect to many enumerations, we refer to Remark \ref{Many} and Section 6 of \cite{Kioulafa}.
	\end{remark}
	
	\pagebreak
	\itshape
	Department of Mathematics\\
	National and Kapodistrian University of Athens\\
	Panepistimiopolis, 15784, Athens, Greece\\[\baselineskip]
	E-mail Addresses:\\
	g\_gavrilos@yahoo.gr\\
	conmaron@gmail.com\\
	vnestor@math.uoa.gr

\begin{thebibliography}{6}
		\bibitem{Abakumov}
		\textsc{E. Abakumov, J. M\"uller} and \textsc{V. Nestoridis}, 'Families of Universal Taylor series depending on a parameter', \textit{New trends in approximation theory}, 201-213 Fields Inst. Commun. 81, Springer, NY 2018.
		
		\bibitem{Bayart}
		\textsc{F. Bayart, K.-G. Grosse-Erdmann, V. Nestoridis} and \textsc{C. Papadimitropoulos}, 'Abstract theory of universal series and applications', \textit{Proc. London Math. Soc.} (3) 96 (2008) 417-463.
		
		\bibitem{Chui}
		\textsc{C. Chui} and \textsc{M. N. Parnes}, 'Approximation by overconvergence of a
		power series', \textit{JMAA} 36 (1971) 693 - 696.
		
		\bibitem{Clouatre}
		\textsc{R. Clou\^atre}, 'Universal power series in $C^N$' \textit{Canad. Math. Bull.} 54 (2011) no 2, 230-236.
		
		\bibitem{Daras}
		\textsc{N. Daras} and \textsc{V. Nestoridis}, 'Universal Taylor series on convex subsets of $\mathbb{C}^N$, Complex Var. Elliptic Equ. 60 (2015) no 11, 1567 - 1579.
		
		\bibitem{Falco}
		\textsc{J. Falco, P. M. Gauthier, M. Manolaki} and \textsc{V. Nestoridis}, 'A function algebra providing new Mergelyan type theorem in several complex variables' arXiv:1901.01339.
		
		\bibitem{Grosse1}
		\textsc{K.-G. Grosse-Erdmann}, 'Holomorphe Monster und universelle Funktionen', \textit{Mitt. Math. Sem. Giessen} 176 (1987), iv + 84 pp.
		
		\bibitem{Grosse2}
		\textsc{K.-G. Grosse-Erdmann}, 'Universal families and hypercyclic operators', \textit{Bull. Amer. Math. Soc.} 36 (1999) no 3, 345 - 381.
		
		\bibitem{Kahane}
		\textsc{J.-P. Kahane}, 'Baire’s Category theorem and trigonometric series', \textit{J. Anal. Math.} 80 (2000), 143 - 182.
		
		\bibitem{Kariofillis}
		\textsc{Ch. Kariofillis}, \textsc{Ch. Konstadilaki} and \textsc{V. Nestoridis}, 'Smooth universal Taylor series', Monatsh. Math. 147 (2006), 249 - 257.
		
		\bibitem{Kioulafa}
		\textsc{K. Kioulafa}, \textsc{G. Kotsovolis} and \textsc{V. Nestoridis}, 'Universal Taylor series on products of planar domains', \textit{CASY}, to appear, see also arxiv:1909.03521.
		
		\bibitem{Luh}
		\textsc{W. Luh}, 'Approximation analytischer Funktionen durch \"uberkonvergente Potenzreihen und deren Matrix - Transformierten', \textit{Mitt. Math. Sem. Giessen} 88 (1970), 1 - 36.
		
		\bibitem{Menshov}
		\textsc{D. Menshov}, 'Sur les s\'eries trigonometriques universelles', \textit{C. R. (Doklady) Acad. Sci. URSS (N. S.)} 49 (1945) 79 - 82.
		
		\bibitem{Nestoridis1}
		\textsc{V. Nestoridis}, 'Universal Taylor series', \textit{Ann. Inst. Fourier} 46 (1996), 1293 - 1306.
		
		\bibitem{Nestoridis2}
		\textsc{V. Nestoridis}, 'A strong notion of universal Taylor series', \textit{J. London Math. Soc.} 68 (2003), 712 - 724.
		
		\bibitem{Pal}
		\textsc{J. Pal}, 'Zwei kleine Bemerkungen', \textit{Tohoku Math. J.} 6 (1914-1915) 42 - 43.
		
		\bibitem{Seleznev}
		\textsc{A. I. Seleznev}, 'On universal power series' (Russian), \textit{Mat. Sbornik (N. S.)} 28 (1951) 453 - 460.
		
		\bibitem{Vlachou}
		\textsc{V. Vlachou}, 'On some classes of universal functions', \textit{Analysis} 22 (2002) 149 - 161.
	\end{thebibliography}
\end{document}